\documentclass[reqno]{amsart}%
\usepackage{amssymb}
\usepackage{graphicx}
\usepackage{amsthm}
\usepackage{amsfonts}
\usepackage{latexsym}
\usepackage{amsmath}
\usepackage{setspace}
\newtheorem{theorem}{Theorem}[section]

\newtheorem{proposition}[theorem]{Proposition}

\newtheorem{lemma}[theorem]{Lemma}

\theoremstyle{definition}
\newtheorem{definition}[theorem]{Definition}
\newtheorem{remark}[theorem]{Remark}
\newtheorem{example}[theorem]{Example}
\begin{document}
\title[Nonuniform exponential stability]{Admissible Banach function spaces for linear dynamics with nonuniform behavior
on the half-line}
\author[N. Lupa]{Nicolae Lupa}
\address{N. Lupa, Department of Mathematics, Politehnica University of Timi\c soara,
Victoriei Square 2, 300006 Timi\c soara, Romania}
\email{nicolae.lupa@upt.ro}
\author[L.H. Popescu]{Liviu Horia Popescu}
\address{L. H. Popescu, Department of Mathematics and Informatics, Faculty of Sciences,
University of Oradea, Universit\u a\c tii St. 1, 410087 Oradea, Romania}
\email{lpopescu2002@yahoo.com}

\begin{abstract}
For nonuniform exponentially bounded evolution families on the half-line, we
introduce a special class of Banach function spaces, on which we define
certain $C_{0}$-semigroups. We characterize the existence of nonuniform
exponential stability in terms of invertibility of the corresponding infinitesimal
generators. The invertibility of  these generators is connected to a particular type
of admissible exponents that are specific to nonuniform behavior. For the
bounded orbits, nonuniform exponential stability results from a spectral
property of generators. The $C_{0}$-semigroups we deal with verify the
spectral mapping theorem, as well as the evolution semigroups, in the uniform case. In
particular, our results directly apply to all linear differential equations with
finite Lyapunov exponent.
\end{abstract}
\maketitle

\section{Introduction}

\label{sec.intr}

A linear dynamics is called \emph{well-posed} if we assume the existence,
uniqueness and continuous dependence of solutions on initial data. For a  nonautonomous  linear differential equation on the
half-line $dx/dt=A(t)x$, well-posedness
is equivalent to the existence of an \emph{evolution family} solving the
equation (Proposition 9.3 in \cite{Eng}, p. 478).
In particular, if the linear operators $A\left(  t\right)  $ are bounded, then well-posedness
is guaranteed \cite[Chapter 3]{Dal}.

In the stability theory of linear dynamics, a central problem is to find
conditions for the existence of exponential stability,  dichotomy or trichotomy. A significant method is represented by the input-output techniques,
often called \emph{admissibility methods}. More exactly, the study of
asymptotic behavior of an evolution family $\mathcal{U}=\left\{  U(t,s)\right\}_{t\geq s\geq0}$ reduces to the analysis of the
solvability of the integral equation
\[
u(t)=U(t,s)u(s)+\int_{s}^{t}U(t,\tau)f(\tau)d\tau,\;t\geq s\geq0,
\]
in a wide class of pairs of certain function spaces (ex. $L^{p}(\mathbb{R}%
_{+},X)$, Sch\"{a}ffer spaces, Lorentz spaces or some function spaces occurring in the interpolation
theory \cite{Huy,Pr.Po.Pr.2005}).

Another approach uses the so-called \emph{evolution semigroup} $\mathcal{T}%
=\{T(t)\}_{t\geq0}$ on some appropriate Banach function spaces, basically
defined as
\begin{equation}
(T(t)u)(s)=%
\begin{cases}
U(s,s-t)u(s-t), & \text{ if }s>t,\\
U(s,0)u(0), & \text{ if }0\leq s\leq t.
\end{cases}
\label{eq.semigroup}%
\end{equation}
If the evolution family $\mathcal{U}=\left\{  U(t,s)\right\}_{t\geq s\geq0}$
is \emph{uniform exponentially bounded}, i.e. $\sup\limits_{t\geq s\geq
0}e^{-\alpha(t-s)}\left\Vert U(t,s)\right\Vert <\infty$ for some $\alpha
\in\mathbb{R}$, then $\mathcal{T}$ becomes a $C_{0}$-semigroup on certain
function spaces, as for example $C_{00}(\mathbb{R}_{+},X)$ or $L^{p}%
(\mathbb{R}_{+},X)$, $1\leq p<\infty$. In this case the study of a
nonautonomous equation reduces to the analysis of an autonomous one, precisely
the asymptotic behavior of $\mathcal{U}$ can be characterized in terms of
spectral properties of the generator of $\mathcal{T}$. For more details on
this issue, we refer the reader to monograph \cite{Ch.La.1999} or to papers
\cite{Bu,Mi.Ra.Sc.1998}.

The classical theory of uniform behavior is unfortunately too restrictive, as
linear dynamics mostly fail to exhibit this type of behavior. In the last
decades a more general view emerged: \emph{nonuniform behavior}. A serious
motivation for introducing this concept, not to mention its evident
generality, lies in the ergodic theory (see  for instance \cite{Ba.Pe.2002,Ba.Va.2008-1} and
the references therein). Roughly speaking, while uniformity relates to the
finiteness of the \emph{Bohl exponent} \cite[Chapter 3]{Dal}, nonuniformity
analyses the more general situations when the \emph{Lyapunov exponent} is
finite \cite{Ba.Pe.2002}. Pretty recent works, as for instance
\cite{Ba.Dr.Va,Ba.Va.2010,Me.Sa.Sa.2002-1,Sa.Ba.Sa.2013,chin}, expose significant
admissibility-type results for nonuniform behavior. Let us also mention
another interesting generalization of the classical theory of stability:
asymptotic behavior which is both nonuniform and not necessarily exponential.
For example, paper \cite{Be.Si.2012} studies the nonuniform polynomial behavior.

The main goal of this paper is \emph{to study nonuniform exponential stability
of an evolution family on the half-line using the $C_{0}$-semigroups theory}.
Such endeavor is not at all of a formal type, as illustrated in our examples.
For instance, Example \ref{ex2} presents a uniform exponentially bounded
evolution family, which is not uniform, but nonuniform exponentially stable.
In this case the evolution semigroup exists, but it does not furnish any kind
of information about the asymptotic behavior of the orbits. From another hand,
the evolution family in our Example \ref{ex1} is at the same time nonuniform
exponentially stable, and not uniform exponentially bounded. In this case it
is impossible to construct the evolution semigroup. To sum up, in many
situation the classical tool either does not exist, or it is completely
useless.
Let us briefly present the main ideas of our work.

For a nonuniform exponentially bounded evolution family $\mathcal{U}$ and for
a fixed admissible exponent $\alpha\in\mathbb{R}$, we introduce the
corresponding \emph{admissible Banach function space}, precisely looking for
those functions $u\in C_{00}(\mathbb{R}_{+},X)$, for which the map
$s\mapsto\sup\limits_{t\geq s}e^{-\alpha(t-s)}\left\Vert U(t,s)u\left(
s\right)  \right\Vert $ vanishes at infinity. On each admissible Banach
function space, evidently depending on $\mathcal{U}$ and $\alpha$, we define a
$C_{0}$-semigroup, formally using formula (\ref{eq.semigroup}).  We
completely characterize the existence of nonuniform exponential stability of
$\mathcal{U}$ in terms of invertibility of infinitesimal generators. We give
a necessary and sufficient condition for a fixed generator to be invertible,
introducing a particular type of admissible exponent that we call
\emph{quasi-negative}, specific to nonuniform behavior. Let us emphasize that
the $C_{0}$-semigroups we introduce verify \emph{the spectral mapping theorem}
(as well as the evolution semigroups, in the uniform setting). In the last
section we prove a sufficient condition for the existence of nonuniform
exponential stability of all bounded orbits of a linear dynamics.

\section{Admissible Banach function spaces and $C_{0}$-semigroups}

\label{sec.admissible}

Throughout our paper $X$ is a Banach space, $C(\mathbb{R}_{+},X)$ denotes the
space of all continuous $X$-valued functions defined on the half-line, and
$C_{c,0}(\mathbb{R}_{+},X)$ is the space of all functions in $C(\mathbb{R}%
_{+},X)$ with compact support vanishing at $0$. We also make use of the
following notation:
\[
C_{00}(\mathbb{R}_{+},X)=\left\{  u\in C(\mathbb{R}_{+},X):\,\lim
\limits_{t\rightarrow\infty}u(t)=u(0)=0\right\}  \text{.}%
\]

Let us first recall the notion of evolution family:
\begin{definition}
\label{d.ev} A family of  bounded
linear operators $\mathcal{U}=\left\{  U(t,s)\right\}
_{t\geq s\geq0}$  is called an \emph{evolution family} (on the half-line)  if
\begin{enumerate}
\item[($e_1$)] $U(t,t)=\mathrm{Id}$, $t\geq0$;

\item[($e_2$)] $U(t,\tau)U(\tau,s)=U(t,s)$, $t\geq\tau\geq s\geq0$;

\item[($e_3$)] the map $(t,s)\mapsto U(t,s)x$ is continuous for every $x\in X$.
\end{enumerate}
\end{definition}

\begin{definition}
\label{d.exp.bound} For any fixed $\alpha\in\mathbb{R}$, an evolution family
$\mathcal{U}$ is called \emph{$\alpha$-nonuniform exponentially bounded}, if
there exists a continuous map $M_{\alpha}:\mathbb{R}_{+}\rightarrow(0,\infty)$
such that
\begin{equation}
\parallel U(t,s)\parallel\leq M_{\alpha}(s)e^{\alpha(t-s)}\text{, }t\geq
s\geq0\text{.} \label{eq.exp.bound}%
\end{equation}

If the above estimation holds for some $\alpha<0$, then $\mathcal{U}$ is
called $\alpha$-\emph{nonuniform exponentially stable}. Each $\alpha$
satisfying (\ref{eq.exp.bound}) is called an \emph{admissible exponent}, and
we denote $\mathcal{A}\left(  \mathcal{U}\right)  $ the set of all admissible
exponents. Evidently, for each evolution family $\mathcal{U}$, the set
$\mathcal{A}\left(  \mathcal{U}\right)  $ is either a (semi) infinite
interval, or empty. If $\mathcal{A}\left(  \mathcal{U}\right)  \neq\emptyset$,
then the evolution family $\mathcal{U}$ is called \emph{nonuniform
exponentially bounded}, and if $\mathcal{A}\left(  \mathcal{U}\right)  $
contains negative admissible exponents, we say that $\mathcal{U}$ is
\emph{nonuniform exponentially stable}.

In the above terminology, whenever there exists a \emph{bounded map}
$M_{\alpha}(s)$ satisfying (\ref{eq.exp.bound}) (which is equivalent to the
existence of a constant one), we just replace the term \textquotedblleft
nonuniform\textquotedblright\ with \textquotedblleft uniform\textquotedblright%
. Also, in such cases we call $\alpha$ a \emph{strict} (\emph{admissible})
\emph{exponent}, and we denote $\mathcal{A}_{s}\left(  \mathcal{U}\right)  $
the set of all strict exponents.
\end{definition}

Throughout this work, if not specified, we always assume that $\mathcal{U}%
=\left\{  U(t,s)\right\}  _{t\geq s\geq0}$ is a nonuniform exponentially
bounded evolution family (i.e. $\mathcal{A}(\mathcal{U})$ is nonempty).

\begin{remark}
\label{rem1} Assume that the evolution family $\mathcal{U}$ is reversible
(i.e. $U(t,s)$ is invertible for all $t\geq s\geq0$, and $U(s,t)$%
=$[U(t,s)]^{-1}$). If the Lyapunov exponent $K_{L}$ is finite and not
attained, then $\mathcal{A}\left(  \mathcal{U}\right)  =\left(  K_{L}%
,\infty\right)  $ (in particular $K_{L}=-\infty$ whenever $\mathcal{A}\left(
\mathcal{U}\right)  =\mathbb{R}$). Also if the Bohl exponent $K_{B}$ is finite
and not attained, then $\mathcal{A}_{s}\left(  \mathcal{U}\right)  =\left(
K_{B},\infty\right)  $. The intervals of admissibility are closed at their
left endpoints whenever the Lyapunov or the Bohl exponents are attained.
\end{remark}

Indeed, as
\[
K_{L}=\inf\left\{  \alpha\in\mathbb{R}:\,\text{there exists }M_{\alpha
}>0\text{ with }\left\Vert U\left(  t,0\right)  \right\Vert \leq M_{\alpha
}e^{\alpha t}\text{, }t\geq0\right\}  ,
\]
if $\alpha\in\mathcal{A}\left(  \mathcal{U}\right)  $, replacing $s=0$ in
\eqref{eq.exp.bound}, one has $\parallel U(t,0)\parallel\leq M_{\alpha
}(0)e^{\alpha t},$ and since $K_{L}$ is not attained, we have $\alpha\in
(K_{L},\infty)$. For $\alpha\in(K_{L},\infty)$, assuming that $\mathcal{U}$ is
reversible, we get
\[
\left\Vert U(t,s)\right\Vert =\left\Vert U(t,0)U(0,s)\right\Vert \leq
M_{\alpha}e^{\alpha t}\left\Vert U(0,s)\right\Vert =M_{\alpha}e^{\alpha
s}\left\Vert U(0,s)\right\Vert e^{\alpha(t-s)},
\]
that implies $\alpha\in\mathcal{A}\left(  \mathcal{U}\right)  $. The second
statement can be proved similarly if we notice that
\[
K_{B}=\inf\left\{  \alpha\in\mathbb{R}:\,\text{there exists }M_{\alpha
}>0\text{ with }\left\Vert U\left(  t,s\right)  \right\Vert \leq M_{\alpha
}e^{\alpha(t-s)},\;t\geq s\geq0\right\}  .
\]

\bigskip

Let $\alpha\in\mathcal{A}\left(  \mathcal{U}\right)  $. For $t\geq0$ and
$u\in{C}(\mathbb{R}_{+},X)$ we set%
\begin{equation}
\varphi_{\mathcal{U},\alpha}(t,u)=\underset{\tau\geq t}{\sup}\text{
}e^{-\alpha(\tau-t)}\parallel U(\tau,t)u(t)\parallel. \label{eq6}%
\end{equation}

If in particular $u(t)\equiv x$ for some $x\in X$, we step over the norm on
$X$ defined in \cite[Eq. (5)]{Ba.Va.2010}, precisely $\left\Vert x\right\Vert
_{t}=\underset{\tau\geq t}{\sup}$ $e^{-\alpha(\tau-t)}\parallel U(\tau
,t)x\parallel$. In this regard, we notice that the map $\varphi_{\mathcal{U}%
,\alpha}$ in (\ref{eq6}) can be (indirectly) defined as $\varphi
_{\mathcal{U},\alpha}(t,u)=\left\Vert u(t)\right\Vert _{t}$.

Inequality \eqref{eq.exp.bound} implies%
\begin{equation}
\parallel u(t)\parallel\leq\varphi_{\mathcal{U},\alpha}(t,u)\leq M_{\alpha
}(t)\parallel u(t)\parallel. \label{eq.init}%
\end{equation}

The following result is essential in the sequel.

\begin{proposition}
\label{p2}The map $\mathbb{R}_{+}\ni t\mapsto\varphi_{\mathcal{U},\alpha
}(t,u)\in\mathbb{R}_{+}$ is continuous for any fixed $u\in{C}(\mathbb{R}%
_{+},X)$. In addition, for every $u\in{C}(\mathbb{R}_{+},X)$ for which
$\lim\limits_{t\rightarrow\infty}\varphi_{\mathcal{U},\alpha}(t,u)=0$, there
exists (possibly not unique) $t_{u}\geq0$ such that
\[
\sup\limits_{t\geq0}\varphi_{\mathcal{U},\alpha}(t,u)=\varphi_{\mathcal{U}%
,\alpha}(t_{u},u).
\]

\end{proposition}

\begin{proof}
To prove the first statement we set $V(t,s)=e^{-\alpha(t-s)}U(t,s)$. It
follows that $\mathcal{V}=\left\{  V(t,s)\right\}  _{t\geq s\geq0}$ is also an
evolution family with $\parallel V(t,s)\parallel\leq M_{\alpha}(s)$, $t\geq
s\geq0$. For fixed $u\in{C}(\mathbb{R}_{+},X)$, $t_{0}\geq0$ and
$\varepsilon>0$, there exists $\delta_{1}, \delta_{2}>0$ such that
\[
\left\vert t-t_{0}\right\vert <\delta_{1}\Rightarrow M_{\alpha}\left(
t\right)  \left\Vert u\left(  t\right)  -u\left(  t_{0}\right)  \right\Vert
<{\varepsilon}/{3}\text{,}%
\]%
\[
t_{0}\leq t<t_{0}+\delta_{2}\Rightarrow M_{\alpha}\left(  t\right)  \left\Vert
u\left(  t_{0}\right)  -V\left(  t,t_{0}\right)  u\left(  t_{0}\right)
\right\Vert <{\varepsilon}/{3}\text{.}%
\]
Let $\delta=\max\left\{  \delta_{1},\delta_{2}\right\}  $ and choose $t\geq0$
with $\left\vert t-t_{0}\right\vert <\delta$. We analyze the case $t\geq
t_{0}$. For any $\tau\geq t$ we have%
\begin{align*}
\left\Vert V(\tau,t)u\left(  t\right)  \right\Vert  &  \leq\left\Vert
V(\tau,t)\left(  u\left(  t\right)  -u\left(  t_{0}\right)  \right)
\right\Vert +\left\Vert V(\tau,t)u\left(  t_{0}\right)  -V(\tau,t_{0})u\left(
t_{0}\right)  \right\Vert \\
&  \quad+\left\Vert V(\tau,t_{0})u\left(  t_{0}\right)  \right\Vert \\
&  \leq\left\Vert V(\tau,t)\right\Vert \left\Vert u\left(  t\right)  -u\left(
t_{0}\right)  \right\Vert +\left\Vert V(\tau,t)\right\Vert \left\Vert u\left(
t_{0}\right)  -V\left(  t,t_{0}\right)  u\left(  t_{0}\right)  \right\Vert \\
&  \quad+\left\Vert V(\tau,t_{0})u\left(  t_{0}\right)  \right\Vert \\
&  \leq M_{\alpha}(t)\left\Vert u\left(  t\right)  -u\left(  t_{0}\right)
\right\Vert +M_{\alpha}(t)\left\Vert u\left(  t_{0}\right)  -V\left(
t,t_{0}\right)  u\left(  t_{0}\right)  \right\Vert \\
&  \quad+\varphi_{\mathcal{U},\alpha}(t_{0},u)\\
&  \leq2\varepsilon/3+\varphi_{\mathcal{U},\alpha}(t_{0},u).
\end{align*}
Taking the supremum with respect to $\tau\geq t$ we get
\[
\varphi_{\mathcal{U},\alpha}(t,u)-\varphi_{\mathcal{U},\alpha}(t_{0}%
,u)<\varepsilon.
\]
The case $t<t_{0}$ results using the same type of arguments. Similarly one can
prove that $\varphi_{\mathcal{U},\alpha}(t_{0},u)-\varphi_{\mathcal{U},\alpha
}(t,u)<\varepsilon$, hence the map $t\mapsto\varphi_{\mathcal{U},\alpha}(t,u)$
is continuous at $t_{0}$ for arbitrary $t_{0}\geq0$, which proves the first
statement. The second one follows from the continuity of the map in question,
together with condition $\lim\limits_{t\rightarrow\infty}\varphi
_{\mathcal{U},\alpha}(t,u)=0$.
\end{proof}

For each $\alpha\in\mathcal{A}\left(  \mathcal{U}\right)  $ we set%
\[
\mathcal{C}(\mathcal{U},\alpha)=\left\{  u\in C(\mathbb{R}_{+},X):\,\lim
\limits_{t\rightarrow\infty}\varphi_{\mathcal{U},\alpha}(t,u)=\|u(0)\|=0\right\}
.
\]
Eq. \eqref{eq.init} implies
\[
C_{c,0}(\mathbb{R}_{+},X)\subset\mathcal{C}(\mathcal{U},\alpha)\subset
C_{00}(\mathbb{R}_{+},X), \label{eq.ccc}%
\]
and thus $\mathcal{C}(\mathcal{U},\alpha)$ is nonempty whenever $\alpha
\in\mathcal{A}(\mathcal{U})$. Furthermore, one can show that $\mathcal{C}%
(\mathcal{U},\alpha)$ is a Banach function space equipped with the norm
\[
\parallel u\parallel_{\mathcal{U},\alpha}=\sup\limits_{t\geq0}\text{ }%
\varphi_{\mathcal{U},\alpha}(t,u),
\]
and we call it the \emph{admissible Banach function space} corresponding to
the evolution family $\mathcal{U}$ and the admissible exponent $\alpha
\in\mathcal{A}(\mathcal{U})$.

Let us remark that if $\alpha\in\mathcal{A}\left(  \mathcal{U}\right)  $, and
$\beta\geq\alpha$, then $\beta\in\mathcal{A}\left(  \mathcal{U}\right)  $.
Moreover,
\[
\mathcal{C}(\mathcal{U},\alpha)\subset\mathcal{C}\left(  \mathcal{U}%
,\beta\right)  \text{ and }\parallel u\parallel_{\mathcal{U},\beta}%
\leq\parallel u\parallel_{\mathcal{U},\alpha}, \text{ for } u\in
\mathcal{C}(\mathcal{U},\alpha).
\]

In the next theorem we extend the notion of evolution semigroup to nonuniform
exponentially bounded evolution families.

Let us first recall  that a family
of bounded linear operators $\mathcal{T}=\left\{  T(t)\right\}  _{t\geq0}$
acting on a Banach space $X$ is a $C_{0}$-semigroup if
\begin{enumerate}
\item[($s_{1}$)] $T(0)=\mathrm{Id}$;

\item[($s_{2}$)] $T(t)T(s)=T(t+s)$ for $t,s\geq0$;

\item[($s_{3}$)] $\lim\limits_{t\rightarrow0_{+}}T(t)x=x$ for every $x\in X$.
\end{enumerate}
The (closed and densely defined) linear operator
$G:D(G)\subset X\rightarrow X$, where $Gx=\lim\limits_{t\rightarrow0_{+}}%
\frac{T(t)x-x}{t}, $ is called the (infinitesimal) \emph{generator} of the
$C_{0}$-semigroup $\mathcal{T}$.

\begin{theorem}
\label{th3}Each $\alpha\in\mathcal{A}\left(  \mathcal{U}\right)  $ defines a
$C_{0}$-semigroup $\mathcal{T}_{\alpha}=\left\{  T_{\alpha}(t)\right\}
_{t\geq0}$ on $\mathcal{C}(\mathcal{U},\alpha)$ by setting
\begin{equation}
(T_{\alpha}(t)u)(s)=%
\begin{cases}
U(s,s-t)u(s-t)\text{,} & \text{if }s>t\text{,}\\
0\text{,} & \text{if }0\leq s\leq t\text{.}%
\end{cases}
\label{eq2}%
\end{equation}
Moreover, the following estimation holds%
\begin{equation}
\parallel T_{\alpha}(t)u\parallel_{\mathcal{U},\alpha}\leq e^{\alpha
t}\parallel u\parallel_{\mathcal{U},\alpha}, \;u\in\mathcal{C}(\mathcal{U}%
,\alpha), \; t\geq0\text{.} \label{eq3}%
\end{equation}

\end{theorem}

\begin{proof}
Evidently $T_{\alpha}(0)=\mathrm{Id}$ and $T_{\alpha}(t)T_{\alpha
}(s)=T_{\alpha}(t+s)$, for all $t,s\geq0$. It remains to prove that the map
$T_{\alpha}(t)$ in (\ref{eq2})\ is well defined on $\mathcal{C}(\mathcal{U}%
,\alpha)$, and the norm $\parallel T_{\alpha}(t)u-u\parallel_{\mathcal{U}%
,\alpha}\rightarrow0$ as $t\rightarrow0_{+}$, for each fixed $u\in
\mathcal{C}(\mathcal{U},\alpha)$. Pick $t\geq0$ and $u\in\mathcal{C}%
(\mathcal{U},\alpha)$. For arbitrary $s\geq t$ we have
\begin{align*}
\varphi_{\mathcal{U},\alpha}(s,T_{\alpha}(t)u)  &  =\sup\limits_{\tau\geq
s}e^{-\alpha(\tau-s)}\parallel U(\tau,s-t)u(s-t)\parallel\\
&  =e^{\alpha t}\sup\limits_{\tau\geq s}e^{-\alpha\lbrack\tau-(s-t)]}\parallel
U(\tau,s-t)u(s-t)\parallel\\
&  \leq e^{\alpha t}\varphi_{\mathcal{U},\alpha}(s-t,u),
\end{align*}
therefore $\varphi_{\mathcal{U},\alpha}(s,T_{\alpha}(t)u)\rightarrow0$ as
$s\rightarrow\infty$, which leads to $T_{\alpha}(t)u\in\mathcal{C}%
(\mathcal{U},\alpha)$. Notice that the above estimation also proves inequality
(\ref{eq3}).

We now claim that the space $C_{c,0}(\mathbb{R}_{+},X)$ is dense in
$\mathcal{C}(\mathcal{U},\alpha)$ with respect to the norm $\parallel
\cdot\parallel_{\mathcal{U},\alpha}$. For any fixed $u\in\mathcal{C}%
(\mathcal{U},\alpha)$ and any non-negative integer $n\in\mathbb{N}$, let us
consider a continuous function $\alpha_{n}:\mathbb{R}_{+}\rightarrow
\lbrack0,1]$ such that
\[
\alpha_{n}(t)=1,\text{ for all }t\in\lbrack0,n],\text{ and }\alpha
_{n}(t)=0,\text{ for all }t\geq n+1.
\]
Putting $u_{n}=\alpha_{n}u$, we notice that $u_{n}\in C_{c,0}(\mathbb{R}%
_{+},X)$. We claim that  $$\underset{n\rightarrow\infty}{\lim
}\parallel u_{n}-u\parallel_{\mathcal{U},\alpha}=0.$$ Indeed, as $\,\lim
\limits_{t\rightarrow\infty}\varphi_{\mathcal{U},\alpha}(t,u)=0$, it follows
that for each $\varepsilon>0$ there exists $\delta>0$ such that for $t>\delta$
we have $\varphi_{\mathcal{U},\alpha}(t,u)<\varepsilon/2$. Set $n_{0}%
=[\delta]+1$ and choose $n\geq n_{0}$. The definition of the map $\alpha_{n}$
readily implies
\[
\parallel u_{n}-u\parallel_{\mathcal{U},\alpha}=\sup\limits_{t\geq n}%
\varphi_{\mathcal{U},\alpha}(t,u_{n}-u)\leq\sup\limits_{t\geq n}%
\varphi_{\mathcal{U},\alpha}(t,u)<\varepsilon,
\]
which concludes the claim.

For the second statement, pick $u\in C_{c,0}(\mathbb{R}_{+},X)$. There exist
$a$, $b\geq0$, $a<b$ such that $supp(T_{\alpha}(t)u-u)\subset\lbrack a,b]$,
for sufficiently small $t\geq0$. For such $t$ we have
\begin{align*}
\parallel T_{\alpha}(t)u-u\parallel_{\mathcal{U},\alpha}  &  =\sup
\limits_{s\geq0}\varphi_{\mathcal{U},\alpha}(s,T_{\alpha}(t)u-u)\\
&  \leq\sup\limits_{s\geq0}M_{\alpha}(s)\parallel(T_{\alpha}%
(t)u)(s)-u(s)\parallel\\
&  =\sup\limits_{s\in supp(T_{\alpha}(t)u-u)}M_{\alpha}(s)\parallel(T_{\alpha
}(t)u)(s)-u(s)\parallel\\
&  \leq K_{\alpha}\sup\limits_{s\in supp(T_{\alpha}(t)u-u)}\parallel
(T_{\alpha}(t)u)(s)-u(s)\parallel,
\end{align*}
where $K_{\alpha}=\max\limits_{s\in\lbrack a,b]}M_{\alpha}(s)$. Using standard
arguments (ex. \cite{Rau}), one can easily prove that $\parallel T_{\alpha
}(t)u-u\parallel_{\mathcal{U},\alpha}\rightarrow0$ as $t\rightarrow0_{+}$, and
this completes the proof.
\end{proof}

The $C_{0}$-semigroup $\mathcal{T}_{\alpha}$ is the analogue of the evolution
semigroup on $C_{00}(\mathbb{R}_{+},X)$ corresponding to an $\alpha
$-nonuniform exponentially bounded evolution family. We denote $G_{\mathcal{U}%
,\alpha}$ the generator of $\mathcal{T}_{\alpha}$.

\begin{remark}
\label{rem.gen} If the Banach function spaces $\mathcal{C}(\mathcal{U}%
,\alpha)$ and $\mathcal{C}(\mathcal{U},\beta)$ coincide for some admissible
exponents $\alpha,\beta\in\mathcal{A}(\mathcal{U})$, then $\mathcal{T}%
_{\alpha}=\mathcal{T}_{\beta}$ and $G_{\mathcal{U},\alpha}=G_{\mathcal{U}%
,\beta}$.
\end{remark}

Next proposition illustrates the connection between the $C_{0}$-semigroup
$\mathcal{T}_{\alpha}$ and the evolution semigroup on $C_{00}(\mathbb{R}%
_{+},X)$ defined in (\ref{eq.semigroup}).

\begin{proposition}
\label{p1} For $\alpha\in\mathcal{A}\left(  \mathcal{U}\right)  $, the Banach
spaces $\mathcal{C}(\mathcal{U},\alpha)$ and $C_{00}(\mathbb{R}_{+},X)$
coincide if and only if $\alpha$ is a strict exponent. In this case, the
$C_{0}$-semigroup $\mathcal{T}_{\alpha}$ coincides with the evolution
semigroup $\mathcal{T}$ on $C_{00}(\mathbb{R}_{+},X)$.
\end{proposition}

\begin{proof}
\emph{Necessity}. Let us assume that $\mathcal{C}(\mathcal{U},\alpha
)=C_{00}(\mathbb{R}_{+},X)$, for some $\alpha\in\mathcal{A}\left(
\mathcal{U}\right)  $. If $\alpha$ is not a strict exponent, then for each
positive integer $n\in\mathbb{N}^{\ast}$ there exist $t_{n}\geq s_{n}\geq0$
and $x_{n}\in X$ with $\parallel x_{n}\parallel=1$, such that
\[
\parallel U(t_{n},s_{n})x_{n}\parallel>ne^{\alpha\left(  t_{n}-s_{n}\right)
}.
\]
If the sequence $\left(  s_{n}\right)  _{n\in\mathbb{N}^{\ast}}$ is bounded,
say $s_{n}\leq k$, then inequality
\[
n<e^{-\alpha\left(  t_{n}-s_{n}\right)  }\parallel U(t_{n},s_{n}%
)x_{n}\parallel\leq M_{\alpha}\left(  s_{n}\right)  \leq\underset{0\leq s\leq
k}{\sup}M_{\alpha}\left(  s\right)
\]
leads to a contradiction, $n<\underset{0\leq s\leq k}{\sup}M_{\alpha}\left(
s\right)  $ for all $n\in\mathbb{N}^{\ast}$. Thus, without loss of generality
one can always assume that the sequence $\left(  s_{n}\right)  _{n\in
\mathbb{N}^{\ast}}$ is strictly increasing, unbounded and let us put
$t_{0}=s_{0}=0$, $x_{0}=0$. Setting $y_{n}=\frac{x_{n}}{\sqrt{n}}$, $n\geq1$
and $y_{0}=0$, one gets
\[
e^{-\alpha\left(  t_{n}-s_{n}\right)  }\parallel U(t_{n},s_{n})y_{n}%
\parallel\geq\sqrt{n}\text{, }n\in\mathbb{N}\text{.}\
\]
Consider $u_{y}:\mathbb{R}_{+}\rightarrow X$ by
\[
u_{y}\left(  s\right)  =\frac{s\left(  y_{n+1}-y_{n}\right)  }{s_{n+1}-s_{n}%
}+\frac{s_{n+1}y_{n}-s_{n}y_{n+1}}{s_{n+1}-s_{n}},\text{ if }s\in\left[
s_{n},s_{n+1}\right]  ,n\in\mathbb{N}\text{.}%
\]
We notice that $u_{y}\left(  s_{n}\right)  =$ $y_{n}$ for each $n\in
\mathbb{N}$,\ that results in $u_{y}\in C_{00}(\mathbb{R}_{+},X)$. From
estimation%
\[
\varphi_{\mathcal{U},\alpha}(s_{n},u_{y})\geq e^{-\alpha\left(  t_{n}%
-s_{n}\right)  }\parallel U(t_{n},s_{n})y_{n}\parallel\geq\sqrt{n},
\]
we deduce $u_{y}\notin\mathcal{C}(\mathcal{U},\alpha)$, that is $\mathcal{C}%
(\mathcal{U},\alpha)\neq C_{00}(\mathbb{R}_{+},X)$, which is a contradiction.

\emph{Sufficiency}. If $\alpha$ is a strict exponent, then there exists
$M_{\alpha}>0$ such that $$\parallel U(t,s)\parallel\leq M_{\alpha}%
e^{\alpha(t-s)}, \text{ for }t\geq s\geq0.$$ Assumption $u\in C_{00}(\mathbb{R}%
_{+},X)$ implies $\varphi_{\mathcal{U},\alpha}(t,u)\leq M_{\alpha}\parallel
u(t)\parallel\rightarrow0$ as $t\rightarrow\infty$, hence $u\in\mathcal{C}%
(\mathcal{U},\alpha)$.
\end{proof}

In the next example we identify the set $\mathcal{A}(\mathcal{U})$, for a
given evolution family $\mathcal{U}$. We point out the connections between the
admissible Banach function spaces $\mathcal{C}(\mathcal{U},\alpha)$, when
$\alpha$ varies on $\mathcal{A}(\mathcal{U})$. Let us emphasize that in this
particular case there are infinitely many admissible Banach function spaces.

\begin{example}
\label{ex2}For $t\geq s>0$ we set
\[
E\left(  t,s\right)  =s\left(\sqrt{2}+\sin\ln s\right)  -t\left(  \sqrt
{2}+\sin\ln t\right),
\]
$E(t,0)=-t\left(\sqrt{2}+\sin\ln t\right)$ for $t>0$, and $E(0,0)=0$.
We consider the evolution family $U\left(  t,s\right)  =e^{E\left(
t,s\right)  }\mathrm{Id}$. We claim that $\mathcal{U}$ has the following properties:

\begin{enumerate}
\item $\mathcal{A}\left(  \mathcal{U}\right)  =\left[  1-\sqrt{2}%
,\infty\right)  $;

\item $\mathcal{U}$ is $\alpha$-uniform exponentially bounded, if $\alpha
\geq0$;

\item $\mathcal{C}(\mathcal{U},\alpha)=C_{00}(\mathbb{R}_{+},X)$, if
$\alpha\geq0$;

\item $\mathcal{U}$ is not $\alpha$-uniform exponentially bounded, if
$\alpha\in\left[  1-\sqrt{2},0\right)  $;

\item $\mathcal{C}(\mathcal{U},\alpha)\neq\mathcal{C}(\mathcal{U},\beta)$, for
all different $\alpha,\beta\in\lbrack1-\sqrt{2},0)$.
\end{enumerate}
\end{example}

\begin{proof}
Since $\lim\limits_{s\searrow0}s\sin\ln s=0$, we make the convention  $0\,\sin\ln 0=0$.
For any fixed $\alpha\in\mathbb{R}$, let us denote
\begin{align*}
E_{\alpha}\left(  t,s\right) &=E\left(  t,s\right)  -\alpha\left(  t-s\right)\\
&=s\left(  \alpha+\sqrt{2}+\sin\ln s\right)  -t\left(  \alpha+\sqrt{2}+\sin\ln
t\right),\;t\geq s\geq 0  \text{.}%
\end{align*}
If $\alpha<1-\sqrt{2}$, then $\alpha=1-\sqrt{2}-a$, for some $a>0$. In this
case we have
\begin{align*}
E_{\alpha}\left(  t,s\right)  =s\left(  1-a+\sin\ln s\right)  -t\left(
1-a+\sin\ln t\right), \;t\geq s\geq0.
\end{align*}
Setting $t=e^{2n\pi+\frac{3\pi}{2}}$, $n\in\mathbb{N}$, and $s=0$ we get
$E_{\alpha}(t,s) = e^{2n\pi+\frac{3\pi}{2}} a \rightarrow\infty$ as
$n\to\infty$, which implies that $\alpha\notin\mathcal{A}\left(
\mathcal{U}\right)  $, and thus $\mathcal{A}\left(  \mathcal{U}\right)
\subset[1-\sqrt{2},\infty)$. On the other hand, we have
\[
E_{1-\sqrt{2}}(t,s)=s (1+\sin\ln{s})-t(1+\sin\ln{t})\leq s (1+\sin\ln{s}),\;t\geq s\geq 0.
\]
Therefore, $1-\sqrt{2}\in\mathcal{A}\left(  \mathcal{U}\right)  $, and
consequently $[1-\sqrt{2},\infty)\subset\mathcal{A}\left(  \mathcal{U}\right)
$. We conclude that $\mathcal{A}\left(  \mathcal{U}\right)  =\left[
1-\sqrt{2},\infty\right)  $.

For $\alpha\in\mathcal{A}(\mathcal{U})$ let us put $f_{\alpha}\left(
t\right)  =t\left(  \alpha+\sqrt{2}+\sin\ln t\right)  $, $t\geq0$. The
derivative of $f_{\alpha}$ is
\[
f_{\alpha}^{\prime}\left(  t\right)  =\alpha+\sqrt{2}\left[  1+\sin\left(
\frac{\pi}{4}+\ln t\right)  \right]  .
\]

If $\alpha\geq0$, then $f_{\alpha}^{\prime}\left(  t\right)  \geq0$,
therefore
\[
E_{\alpha}\left(  t,s\right)  =f_{\alpha}\left(  s\right)  -f_{\alpha}\left(
t\right)  =f_{\alpha}^{\prime}\left(  \theta_{t,s}\right)  \left(  s-t\right)
\leq0,\text{ for }t>s\geq0,
\]
and so $\left\Vert U\left(  t,s\right)  \right\Vert \leq e^{\alpha\left(
t-s\right)  }$. It follows that $\mathcal{U}$ is $\alpha$-uniform
exponentially bounded. Proposition \ref{p1} implies now that $\mathcal{C}%
(\mathcal{U},\alpha)=C_{00}(\mathbb{R}_{+},X)$, for $\alpha\geq0$.

If $\alpha\in\left[  1-\sqrt{2},0\right)  $, then
\[
f_{\alpha}^{\prime}\left(  t\right)  =\sqrt{2}\left[  \sin\theta+\sin\left(
\frac{\pi}{4}+\ln t\right)  \right]  ,
\]
where $\theta=\arcsin\left(  \frac{\sqrt{2}+\alpha}{\sqrt{2}}\right)
\in\left[  \frac{\pi}{4},\frac{\pi}{2}\right)  $. Equation $f_{\alpha}%
^{\prime}\left(  t\right)  =0$ generates two families of solutions:
\[
t_{k}^{\prime}=e^{2k\pi-\frac{\pi}{4}-\theta}\text{ and }t_{k}^{\prime\prime
}=e^{\left(  2k+1\right)  \pi-\frac{\pi}{4}+\theta},\;k\in\mathbb{N}.
\]
Notice that $t_{k}^{\prime}<e^{2k\pi}<t_{k}^{\prime\prime}$. As $f_{\alpha
}^{\prime}\left(  e^{2k\pi}\right)  =\sqrt{2}\left(  \sin\theta+\sin\frac{\pi
}{4}\right)  >0$, the map $f_{\alpha}$ is increasing on $(t_{k}^{\prime}%
,t_{k}^{\prime\prime})$. Notice that for all $\eta$ with $\frac{\pi}{4}%
\leq\theta<\eta<\frac{\pi}{2}$, we have $t_{k}^{\prime\prime}<e^{\left(
2k+1\right)  \pi-\frac{\pi}{4}+\eta}<t_{k+1}^{\prime}$. Since $f_{\alpha
}^{\prime}\left(  e^{\left(  2k+1\right)  \pi-\frac{\pi}{4}+\eta}\right)
=\sqrt{2}\left(  \sin\theta-\sin\eta\right)  <0$, the map $f_{\alpha}$ is
decreasing on $(t_{k}^{\prime\prime},t_{k+1}^{\prime})$. To sum up,
$f_{\alpha}$ is increasing on $(t_{k}^{\prime},t_{k}^{\prime\prime})$ and
decreasing on $(t_{k}^{\prime\prime},t_{k+1}^{\prime})$, for each fixed
$k\in\mathbb{N}$.

Let us denote $s_{k}=e^{\left(  2k+1\right)  \pi+\frac{\pi}{4}}$. Then
$s_{k}\in\left(  t_{k}^{\prime\prime},t_{k+1}^{\prime}\right)  $, and we
deduce that $f_{\alpha}$ is decreasing on $[s_{k},t_{k+1}^{\prime})$.
Observing that the sequence $\left\{  f_{\alpha}(t_{k}^{\prime})\right\}
_{k\in\mathbb{N}}$ is increasing, we conclude that $\underset{t\geq s_{k}%
}{\inf}f_{\alpha}(t)=f_{\alpha}\left(  t_{k+1}^{\prime}\right)  $ and this
yields%
\begin{align*}
\underset{t\geq s_{k}}{\sup}E_{\alpha}\left(  t,s_{k}\right)   &  =f_{\alpha
}(s_{k})-\underset{t\geq s_{k}}{\inf}f_{\alpha}(t)=f_{\alpha}(s_{k}%
)-f_{\alpha}\left(  t_{k+1}^{\prime}\right) \\
&  =f_{\alpha}\left(  e^{\left(  2k+1\right)  \pi+\frac{\pi}{4}}\right)
-f_{\alpha}\left(  e^{(2k+1)\pi+\frac{3\pi}{4}-\theta}\right) \\
&  =\sqrt{2}e^{\left(  2k+1\right)  \pi+\frac{\pi}{4}}\left[  \sin
\theta\left(  1-\frac{1}{2}e^{\frac{\pi}{2}-\theta}\right)  +\frac{1}{2}%
\cos\theta e^{\frac{\pi}{2}-\theta}-\frac{1}{2}\right]  .
\end{align*}
If we put $\varphi\left(  \theta\right)  =\sin\theta\left(  1-\frac{1}%
{2}e^{\frac{\pi}{2}-\theta}\right)  +\frac{1}{2}\cos\theta e^{\frac{\pi}%
{2}-\theta}-\frac{1}{2}$, $\theta\in\left[  \frac{\pi}{4},\frac{\pi}%
{2}\right)  $, then the derivative $\varphi^{\prime}\left(  \theta\right)
=\cos\theta\left(  1-e^{\frac{\pi}{2}-\theta}\right)  <0$, hence
$\varphi\left(  \theta\right)  >\lim\limits_{\xi\rightarrow\pi/2}%
\varphi\left(  \xi\right)  =0$. We deduce that
\[
\underset{k\rightarrow\infty}{\lim}\,\underset{t\geq s_{k}}{\sup}E_{\alpha
}\left(  t,s_{k}\right)  =\underset{k\rightarrow\infty}{\lim}\sqrt
{2}e^{\left(  2k+1\right)  \pi+\frac{\pi}{4}}\varphi\left(  \theta\right)
=\infty.
\]
We conclude that the evolution family $\mathcal{U}$ is not $\alpha$-uniform
exponentially bounded, and thus $\mathcal{C}(\mathcal{U},\alpha)\neq
C_{00}(\mathbb{R}_{+},X)$, for $\alpha\in\left[  1-\sqrt{2},0\right)  $.

Picking $1-\sqrt{2}\leq\alpha<\beta<0$, we denote
\[
\theta=\arcsin\left(  \frac{\sqrt{2}+\alpha}{\sqrt{2}}\right)  \text{ and
}\zeta=\arcsin\left(  \frac{\sqrt{2}+\beta}{\sqrt{2}}\right)  .
\]
Notice that $\frac{\pi}{4}\leq\theta<\zeta<\frac{\pi}{2}$. For any $s\geq0$,
we define $g_{\alpha}\left(  s\right)  =\underset{t\geq s}{\sup}E_{\alpha
}\left(  t,s\right)  <\infty$, and $g_{\beta}\left(  s\right)  =\underset
{t\geq s}{\sup}E_{\beta}\left(  t,s\right)  <\infty$. Set $g\left(  s\right)
=g_{\alpha}\left(  s\right)  -g_{\beta}\left(  s\right)  $. Since
\begin{align*}
g\left(  s_{k}\right)   &  =\underset{t\geq s_{k}}{\sup}E_{\alpha}\left(
t,s_{k}\right)  -\underset{t\geq s_{k}}{\sup}E_{\beta}\left(  t,s_{k}\right)
\\
&  =\sqrt{2}e^{\left(  2k+1\right)  \pi+\frac{\pi}{4}}\left[  \varphi\left(
\theta\right)  -\varphi\left(  \zeta\right)  \right]  \rightarrow\infty\text{
as }k\rightarrow\infty,
\end{align*}
it follows that the map $s\mapsto g\left(  s\right)  $ is unbounded.

Let $h$ be a positive function with
\[
h(0)=0,\,\underset{s\rightarrow\infty}{\lim}h\left(  s\right)  =0,\text{ and
}\underset{s\rightarrow\infty}{\lim}h\left(  s\right)  e^{g\left(  s\right)
}=\infty.
\]
We claim that the map $u\left(  s\right)  =h\left(  s\right)  e^{-g_{\beta
}\left(  s\right)  }$ belongs to $\mathcal{C}(\mathcal{U},\beta)$, but not to
$\mathcal{C}(\mathcal{U},\alpha)$. Indeed, we have
\begin{align*}
\varphi_{\mathcal{U},\beta}\left(  s,u\right)   &  =\underset{t\geq s}{\sup
}\,e^{-\beta\left(  t-s\right)  }\left\Vert U\left(  t,s\right)  u\left(
s\right)  \right\Vert =\underset{t\geq s}{\sup}\,e^{E_{\beta}\left(
t,s\right)  }h\left(  s\right)  e^{-g_{\beta}\left(  s\right)  }\\
&  =h\left(  s\right)  \rightarrow0\text{ as }s\rightarrow\infty,
\end{align*}
thus $u\in\mathcal{C}(\mathcal{U},\beta)$. Similarly one gets
\[
\varphi_{\mathcal{U},\alpha}\left(  s,u\right)  =e^{g_{\alpha}\left(
s\right)  }h\left(  s\right)  e^{-g_{\beta}\left(  s\right)  }=h\left(
s\right)  e^{g\left(  s\right)  }\rightarrow\infty\text{ as }s\rightarrow
\infty,
\]
and this yields $u\notin\mathcal{C}(\mathcal{U},\alpha)$. We conclude that the
inclusion $\mathcal{C}(\mathcal{U},\alpha)\subset\mathcal{C}(\mathcal{U}%
,\beta)$ is strict, therefore the evolution family $\mathcal{U}$ has
infinitely many admissible Banach function spaces.
\end{proof}

\bigskip The following result is taken from the classical theory of $C_{0}%
$-semigroups (see, for instance \cite[Lemma 1.3]{Eng}):

\begin{lemma}
\label{lem.semigroup} Let $\mathcal{T}=\left\{  T(t)\right\}  _{t\geq0}$ be a
$C_{0}$-semigroup on a Banach space $E$, and $G$ its generator. If $x,y\in E$,
then $x\in D(G)$ and $Gx=y$ if and only if
\[
T(t)x-x=\int_{0}^{t}T(\xi)yd\xi,\;t\geq0.
\]

\end{lemma}

Let us substitute $E=\mathcal{C}(\mathcal{U},\alpha)$, $x=u$, $y=-f$ and
$G=G_{\mathcal{U},\alpha}$. This rewrites as follows: if $u$, $f\in
\mathcal{C}(\mathcal{U},\alpha)$, then $u\in D(G_{\mathcal{U},\alpha})$ and
$G_{\mathcal{U},\alpha}u=-f$ if and only if
\[
u(t)=\int_{0}^{t}U(t,\xi)f(\xi)d\xi\text{, }t\geq0\text{.} \label{eq.int}%
\]

\section{Nonuniform exponential stability and invertibility of generators}

\label{sec.nonuniform}

In this section we completely characterize nonuniform exponential stability of
an evolution family $\mathcal{U}$, in terms of invertibility of the
infinitesimal generators $G_{\mathcal{U},\alpha}$, $\alpha\in\mathcal{A}%
(\mathcal{U})$. We specify that throughout our paper the term ``invertible
operator'' refers to a linear map which is algebraically invertible, with
bounded inverse. We expose below what we consider to be one of our main
results, that extends for nonuniform behavior the conclusion in Theorem 2.2
from \cite{Mi.Ra.Sc.1998}. The reader will surely notice that our techniques are
of a completely different type. Let us also point out that the methods used in
this bibliographic source cannot apply in the nonuniform setting.

\begin{theorem}
\label{th4}The evolution family $\mathcal{U}$ is nonuniform exponentially
stable if and only if there exists an invertible generator $G_{\mathcal{U}%
,\alpha}$, $\alpha\in\mathcal{A}(\mathcal{U})$. In particular, if
$\mathcal{U}$ is uniform exponentially bounded, then $\mathcal{U}$ is uniform
exponentially stable if and only if the generator $G$ of the evolution
semigroup $\mathcal{T}$ on $C_{00}(\mathbb{R}_{+},X)$ is invertible.
\end{theorem}

\begin{proof}
We prove the first statement of the theorem.

\emph{Necessity}. Assume that $\mathcal{U}$ is $\alpha$-nonuniform
exponentially stable. For any fixed $f\in\mathcal{C}(\mathcal{U},\alpha)$ we
define
\[
u_{f}(t)=\int_{0}^{t}U(t,\xi)f(\xi)d\xi,\;t\geq0.
\]
From $f=0\Rightarrow u_{f}=0$ we deduce that $G_{\mathcal{U},\alpha}$ is a
one-to-one map on $\mathcal{C}(\mathcal{U},\alpha)$. To prove that
$G_{\mathcal{U},\alpha}$ is invertible, one needs to show first that $u_{f}%
\in\mathcal{C}(\mathcal{U},\alpha)$ for each $f\in\mathcal{C}(\mathcal{U}%
,\alpha)$. We estimate
\begin{align*}
\varphi_{\mathcal{U},\alpha}(t,u_{f})  &  \leq\sup\limits_{\tau\geq t}\text{
}e^{-\alpha(\tau-t)}\int_{0}^{t}\parallel U(\tau,\xi)f(\xi)\parallel d\xi\\
&  \leq\sup\limits_{\tau\geq t}\text{ }e^{-\alpha(\tau-t)}\int_{0}%
^{t}e^{\alpha(\tau-\xi)}\varphi_{\mathcal{U},\alpha}(\xi,f)d\xi\\
&  =\int_{0}^{t}e^{\alpha(t-\xi)}\varphi_{\mathcal{U},\alpha}(\xi
,f)d\xi\text{.}%
\end{align*}
We claim that $\lim\limits_{t\rightarrow\infty}\varphi_{\mathcal{U},\alpha
}(t,u_{f})=0$, therefore $u_{f}\in\mathcal{C}(\mathcal{U},\alpha)$. Indeed, it
follows from $\lim\limits_{t\rightarrow\infty}\varphi_{\mathcal{U},\alpha
}(t,f)=0$ that there exits a constant $K>0$ such that
\begin{equation}
\varphi_{\mathcal{U},\alpha}(t,f)\leq K,\text{ for all }t\geq0. \label{eq.l1}%
\end{equation}
Pick $\varepsilon>0$. Assuming that $\alpha<0$, there exits $\delta_{1}>0$
such that
\begin{equation}
\varphi_{\mathcal{U},\alpha}(t,f)<\dfrac{\varepsilon}{2}(-\alpha),\text{ for
all }t>\delta_{1}. \label{eq.l2}%
\end{equation}
The convergence of integral $\int_{0}^{\infty}e^{\alpha\xi}d\xi$ implies the
existence of some $\delta_{2}>0$, such that for all $t^{\prime},t^{\prime
\prime}>\delta_{2}$ one has
\begin{equation}
\int_{t^{\prime}}^{t^{\prime\prime}}e^{\alpha\xi}d\xi<\dfrac{\varepsilon}{2K}.
\label{eq.l3}%
\end{equation}
Let $\delta=\max\{\delta_{1},\delta_{2}\}$ and choose $t>2\delta$. By
\eqref{eq.l1}--\eqref{eq.l3} we have
\begin{align*}
\int_{0}^{t}e^{\alpha(t-\xi)}\varphi_{\mathcal{U},\alpha}(\xi,f)d\xi &
=\int_{0}^{\delta}e^{\alpha(t-\xi)}\varphi_{\mathcal{U},\alpha}(\xi
,f)d\xi+\int_{\delta}^{t}e^{\alpha(t-\xi)}\varphi_{\mathcal{U},\alpha}%
(\xi,f)d\xi\\
&  \leq K\int_{0}^{\delta}e^{\alpha(t-\xi)}d\xi+\dfrac{\varepsilon}{2}%
(-\alpha)\int_{\delta}^{t}e^{\alpha(t-\xi)}d\xi\\
&  =K\int_{t-\delta}^{t}e^{\alpha\xi}d\xi+\dfrac{\varepsilon}{2}(-\alpha
)\int_{0}^{t-\delta}e^{\alpha\xi}d\xi<\varepsilon,
\end{align*}
which proves the claim. We conclude that $G_{\mathcal{U},\alpha}$ is
algebraically invertible and $G_{\mathcal{U},\alpha}^{-1}f=-u_{f}$.
Furthermore, inequality
\[
\varphi_{\mathcal{U},\alpha}(t,u_{f})\leq\int_{0}^{t}e^{\alpha(t-\xi)}%
\varphi_{\mathcal{U},\alpha}(\xi,f)d\xi
\]
yields
\[
\left\Vert u_{f}\right\Vert _{\mathcal{U},\alpha}\leq\left\Vert f\right\Vert
_{\mathcal{U},\alpha}\int_{0}^{t}e^{\alpha(t-\xi)}d\xi\leq-\frac{1}{\alpha
}\left\Vert f\right\Vert _{\mathcal{U},\alpha}\text{,}%
\]
that is $G_{\mathcal{U},\alpha}^{-1}$ is bounded and the following estimation
holds:
\[
\left\Vert G_{\mathcal{U},\alpha}^{-1}f\right\Vert _{\mathcal{U},\alpha}%
\leq-\frac{1}{\alpha}\left\Vert f\right\Vert _{\mathcal{U},\alpha}.
\label{eq4}%
\]

\emph{Sufficiency}. Without loss of generality we may assume that $\alpha
\geq0$. Suppose that the operator $G_{\mathcal{U},\alpha}:D(G_{\mathcal{U}%
,\alpha})\subset\mathcal{C}(\mathcal{U},\alpha)\rightarrow\mathcal{C}%
(\mathcal{U},\alpha)$ is invertible and put
\[
c=c(\alpha)=\parallel G_{\mathcal{U},\alpha}^{-1}\parallel.
\]

For each positive integer $n\in\mathbb{N}^{\ast}$ we denote $\theta_{n}%
=\ln\frac{e^{n}}{e^{n}-1}\rightarrow0$. For fixed $t>s\geq0$ and $n$ large
enough such that $s+\theta_{n}\leq t\leq n$, let us consider a continuous
function $\alpha_{n}:\mathbb{R}_{+}\rightarrow\lbrack0,1]$ with
\[
\alpha_{n}(\xi)=%
\begin{cases}
0, & \text{ if }0\leq\xi\leq s,\\
1, & \text{ if }s+\theta_{n}\leq\xi\leq n,\\
0, & \text{ if }\xi\geq n+\theta_{n}\text{.}%
\end{cases}
\]

\textit{Step 1}. We prove that
\begin{equation}
\parallel U(t,s)\parallel\leq(c\alpha+1)M_{\alpha}(s),\text{ for }t\geq
s\geq0\text{.} \label{eq.st1}%
\end{equation}

Evidently \eqref{eq.st1} holds for $\alpha=0$. Assume now that $\alpha>0$ and
fix $t>s\geq0$, and $x\in X$. For $n$ large enough we define
\[
f_{n}(\xi)=%
\begin{cases}
\alpha_{n}(\xi)e^{-\alpha(\xi-s)}U(\xi,s)x\text{,} & \text{if }\xi>s,\\
0\text{,} & \text{if }0\leq\xi\leq s\text{.}%
\end{cases}
\]
We have $f_{n}\in C_{c,0}(\mathbb{R}_{+},X)$ and thus $f_{n}\in\mathcal{C}%
(\mathcal{U},\alpha)$. We claim that
\[
\parallel f_{n}\parallel_{\mathcal{U},\alpha}\leq M_{\alpha}(s)\parallel
x\parallel.
\]
Indeed, if $\xi\leq s$, as $f_{n}(\xi)=0$, one has $\varphi_{\mathcal{U}%
,\alpha}(\xi,f_{n})=0$, meanwhile for $\xi>s$ the following estimation holds:
\begin{align*}
\varphi_{\mathcal{U},\alpha}(\xi,f_{n})  &  =\underset{\tau\geq\xi}{\sup
}\text{ }e^{-\alpha(\tau-\xi)}\parallel U(\tau,s)x\parallel\alpha_{n}%
(\xi)e^{-\alpha(\xi-s)}\\
&  \leq\underset{\tau\geq\xi}{\sup}\text{ }e^{-\alpha(\tau-s)}\parallel
U(\tau,s)x\parallel\leq M_{\alpha}(s)\parallel x\parallel.
\end{align*}
Putting $u_{n}=G_{\mathcal{U},\alpha}^{-1}(-f_{n})$ one gets
\begin{align*}
u_{n}(t)  &  =\int_{0}^{t}U(t,\xi)f_{n}(\xi)d\xi\\
&  =\int_{s}^{s+\theta_{n}}\alpha_{n}(\xi)e^{-\alpha(\xi-s)}d\xi
\,U(t,s)x+\int_{s+\theta_{n}}^{t}\alpha_{n}(\xi)e^{-\alpha(\xi-s)}%
d\xi\,U(t,s)x\\
&  =I_{n}\,U(t,s)x+\frac{1}{\alpha}\left[  e^{-\alpha\theta_{n}}%
-e^{-\alpha(t-s)}\right]  U(t,s)x\text{.}%
\end{align*}
Here $I_{n}=\int_{s}^{s+\theta_{n}}\alpha_{n}(\xi)e^{-\alpha(\xi-s)}d\xi$.
Inequality $0\leq I_{n}\leq\frac{1}{\alpha}\left(  1-e^{-\alpha\theta_{n}%
}\right)  $ \ leads to
\begin{align*}
\frac{1}{\alpha}e^{-\alpha\theta_{n}}  &  \parallel U(t,s)x\parallel\\
&  \leq\parallel u_{n}(t)\parallel+\frac{1}{\alpha}\left(  1-e^{-\alpha
\theta_{n}}\right)  \parallel U(t,s)x\parallel+\frac{1}{\alpha}e^{-\alpha
(t-s)}\parallel U(t,s)x\parallel\\
&  \leq\parallel u_{n}\parallel_{\mathcal{U},\alpha}+\frac{1}{\alpha}\left(
1-e^{-\alpha\theta_{n}}\right)  \parallel U(t,s)x\parallel+\frac{1}{\alpha
}M_{\alpha}(s)\parallel x\parallel\\
&  \leq c\parallel f_{n}\parallel_{\mathcal{U},\alpha}+\frac{1}{\alpha}\left(
1-e^{-\alpha\theta_{n}}\right)  \parallel U(t,s)x\parallel+\frac{1}{\alpha
}M_{\alpha}(s)\parallel x\parallel\\
&  \leq\frac{c\alpha+1}{\alpha}M_{\alpha}(s)\parallel x\parallel+\frac
{1}{\alpha}\left(  1-e^{-\alpha\theta_{n}}\right)  \parallel U(t,s)x\parallel
\text{.}%
\end{align*}
Now inequality (\ref{eq.st1}) results immediately when letting $n\rightarrow
\infty$.

\textit{Step 2. }For all $k\in\mathbb{N}$ the following holds
\begin{equation}
\parallel U(t,s)\parallel\leq\frac{c^{k}k!}{(t-s)^{k}}(c\alpha+1)M_{\alpha
}(s)\text{, }t>s\geq0\text{.} \label{eq.st2}%
\end{equation}
Step 1 implies that inequality \eqref{eq.st2} holds for $k=0$. Assume that
\eqref{eq.st2} holds for some $k\in\mathbb{N}$. For fixed $t>s\geq0$, $x\in X$
and sufficiently large $n$ we consider
\[
g_{n,k}(\xi)=%
\begin{cases}
\alpha_{n}(\xi)(\xi-s)^{k}U(\xi,s)x\text{,} & \text{if }\xi>s,\\
0\text{,} & \text{if }0\leq\xi\leq s\text{.}%
\end{cases}
\]
Since $g_{n,k}\in C_{c,0}(\mathbb{R}_{+},X)$, it follows that $g_{n,k}%
\in\mathcal{C}(\mathcal{U},\alpha)$. For $0\leq\xi\leq s$ we notice that
$\varphi_{\mathcal{U},\alpha}(\xi,g_{n,k})=0$, and if $\xi>s$ we have
\begin{align*}
\varphi_{\mathcal{U},\alpha}(\xi,g_{n,k})  &  =\sup\limits_{\tau\geq\xi
}e^{-\alpha(\tau-\xi)}\parallel U(\tau,\xi)g_{n,k}(\xi)\parallel\\
&  =\sup\limits_{\tau\geq\xi}e^{-\alpha(\tau-\xi)}\alpha_{n}(\xi)(\xi
-s)^{k}\parallel U(\tau,s)x\parallel\\
&  \leq\sup\limits_{\tau\geq\xi}(\tau-s)^{k}\parallel U(\tau,s)x\parallel\\
&  \leq c^{k}k!(c\alpha+1)M_{\alpha}(s)\parallel x\parallel\text{,}%
\end{align*}
that results in
\[
\parallel g_{n,k}\parallel_{\mathcal{U},\alpha}\leq c^{k}k!(c\alpha
+1)M_{\alpha}(s)\parallel x\parallel.
\]
If $u_{n,k}=G^{-1}(-g_{n,k})$, then%
\begin{align*}
u_{n,k}(t)  &  =\int_{0}^{t}U(t,\xi)g_{n,k}(\xi)d\xi\\
&  =I_{n,k}\,U(t,s)x+\int_{s+\theta_{n}}^{t}(\xi-s)^{k}d\xi\,U(t,s)x\\
&  =I_{n,k}\,U(t,s)x+\frac{1}{k+1}\left[  (t-s)^{k+1}-\theta_{n}^{k+1}\right]
U(t,s)x\text{,}%
\end{align*}
where $I_{n,k}=\int_{s}^{s+\theta_{n}}\alpha_{n}(\xi)(\xi-s)^{k}d\xi\leq
\frac{1}{k+1}\theta_{n}^{k+1}$. Let us estimate%
\begin{align*}
\frac{(t-s)^{k+1}}{k+1}  &  \,\parallel U(t,s)x\parallel\\
&  \leq\parallel u_{n,k}(t)\parallel+\frac{1}{k+1}\theta_{n}^{k+1}\,\parallel
U(t,s)x\parallel+I_{n,k}\parallel U(t,s)x\parallel\\
&  \leq c^{k+1}k!(c\alpha+1)M_{\alpha}(s)\parallel x\parallel+\frac{2}%
{k+1}\theta_{n}^{k+1}\,\parallel U(t,s)x\parallel\text{.}%
\end{align*}
Letting $n\rightarrow\infty$ we deduce that (\ref{eq.st2}) works for $k+1$.

\textit{Step 3.} Pick $\delta\in(0,1)$. Multiplying (\ref{eq.st2}) by
$\delta^{k}$ and summing with respect to $k\in\mathbb{N}$ one easily gets%
\[
\parallel U(t,s)\parallel\leq\frac{c\alpha+1}{1-\delta}M_{\alpha}%
(s)e^{-\frac{\delta}{c}(t-s)},\,t\geq s\geq0,
\]
which proves that $\mathcal{U}$ is nonuniform exponentially stable.

Using Proposition \ref{p1}, the second statement is a simple consequence of
the first one.
\end{proof}

The above theorem gives a criterion for the existence of nonuniform
exponential stability in terms of invertibility of generators $G_{\mathcal{U}%
,\alpha}$, $\alpha\in\mathcal{A}(\mathcal{U})$. One may wonder whether is it
possible to deduce somehow simple conditions for the invertibility of a such
operator. We responding below to this legitimate question, connecting the
invertibility of the generators $G_{\mathcal{U},\alpha}$ to the existence of a
particular type of admissible exponents that we define below.

\begin{definition}
\label{def1}The admissible exponent $\alpha\in$$\mathcal{A}\left(
\mathcal{U}\right)  $ is called \emph{quasi-negative} if $\mathcal{C}%
(\mathcal{U},\alpha)=\mathcal{C}(\mathcal{U},-\nu)$, for some admissible
exponent $-\nu<0$.
\end{definition}

Evidently each negative admissible exponent is quasi-negative. Proposition
\ref{p1} states in fact that the evolution family $\mathcal{U}$ is uniform
exponentially stable if and only if it has strict quasi-negative exponents. In
the nonuniform case the situation is more complicated. In Example \ref{ex2}
the quasi-negative admissible exponents are only the negative ones. In fact,
in this case the evolution family has no admissible exponents both positive
and quasi-negative at the same time. For the evolution family in the next
example the situation is completely different, all admissible exponents (even
the positive ones) are quasi-negative.

\begin{example}
\label{ex1}For $t\geq s\geq0$ we set
\[
E\left(  t,s\right)  =s\left(  2+\sin s\right)  -t\left(  2+\sin t\right)
\]
(with the convention that $0\,\sin\ln 0=0$, as in Example \ref{ex2}), and let $$U\left(  t,s\right)  =e^{E\left(  t,s\right)  }\mathrm{Id}.$$ We claim
that the evolution family $\mathcal{U}$ is nonuniform exponentially stable,
but not uniform exponentially bounded. In fact,
\[
\mathcal{A}\left(  \mathcal{U}\right)  =\left[  -1,\infty\right)  \text{ and }
\mathcal{C}(\mathcal{U},\alpha)=\mathcal{C}\left(  \mathcal{U},-1\right)
\text{ for all } \alpha\in\mathcal{A}\left(  \mathcal{U}\right)  ,
\]
that is all admissible exponents are quasi-negative.
\end{example}

\begin{proof}
Indeed, for any fixed $\alpha\in\mathbb{R}$ let us denote
\[
E_{\alpha}(t,s)=E(t,s) -\alpha(t-s)=s\left(  2+\alpha+\sin s\right)  -t\left(
2+\alpha+\sin t\right)  , \; t\geq s\geq0.
\]
We have
\begin{align*}
E_{-1}\left(  t,s\right)  =s(1+\sin s)-t(1+\sin t) \leq s+s\sin s=f_{1}\left(
s\right)  \text{.}%
\end{align*}
This implies that $\| U(t,s) \|\leq e^{f_{1}(s)} e^{-(t-s)}$, for all $t\geq
s\geq0$, that is $-1\in\mathcal{A}\left(  \mathcal{U}\right)  $, and
consequently $[-1,\infty)\subset\mathcal{A}\left(  \mathcal{U}\right)  $.
Since $E_{-1}(2n\pi+3\pi/2,2n\pi)=2n\pi\to\infty$ as $n\to\infty$, we get that
$\alpha=-1$ is not a strict exponent, and by Proposition \ref{p1} it follows
that $\mathcal{C}\left(  \mathcal{U},-1\right)  \neq C_{00}(\mathbb{R}_{+}%
,X)$. Assume now that $\alpha<-1$ and let $\varepsilon>0$ such that
$\alpha=-1-\varepsilon$. In this case we have
\[
E_{\alpha}\left(  t,s\right)  =E\left(  t,s\right)  +\left(  1+\varepsilon
\right)  \left(  t-s\right)  =s\left(  1-\varepsilon+\sin s\right)  -t\left(
1-\varepsilon+\sin t\right)  .
\]
Fix $s\geq0$. Since
\[
E_{\alpha}\left(  2n\pi+{3\pi}/{2},s\right)  =s\left(  1-\varepsilon+\sin
s\right)  +\varepsilon\left(  2n\pi+{3\pi}/{2}\right)  ,
\]
for $n$ sufficiently large, it follows that $\underset{t\geq s}{\sup}
E_{\alpha}\left(  t,s\right)  =\infty$, and thus $\alpha\notin\mathcal{A}%
\left(  \mathcal{U}\right)  $. We conclude that $\mathcal{A}\left(
\mathcal{U}\right)  =\left[  -1,\infty\right)  $.

Let now $\alpha\in\mathcal{A}\left(  \mathcal{U}\right)  $. For any fixed
$s\geq0$ and $u\in C(\mathbb{R}_{+},X)$ we have%
\begin{align*}
\varphi_{\mathcal{U},\alpha}\left(  s,u\right)   &  =\underset{t\geq s}{\sup
}e^{-\alpha\left(  t-s\right)  }\left\Vert U\left(  t,s\right)  u\left(
s\right)  \right\Vert =\underset{t\geq s}{\sup}e^{E_{\alpha}\left(
t,s\right)  }\parallel u(s) \parallel=e^{f_{2}\left(  s\right)  }\parallel
u\left(  s\right)  \parallel\text{,}%
\end{align*}
where $f_{2}\left(  s\right)  =\underset{t\geq s}{\sup}E_{\alpha}\left(
t,s\right)  <\infty$, and similarly $\varphi_{\mathcal{U},-1}\left(
s,u\right)  =e^{f_{1}\left(  s\right)  }\parallel u\left(  s\right)
\parallel.$ Put $n_{s}=\left[  \frac{s}{2\pi}-\frac{3}{4}\right]  $. It
follows that $t_{s}=2(n_{s}+1)\pi+{3\pi}/{2}\in\left(  s,s+2\pi\right]  $. Let
us remark that
\[
f_{1}\left(  s\right)  -f_{2}\left(  s\right)  \leq f_{1}\left(  s\right)
-E_{\alpha}\left(  t_{s},s\right)  =\left(  1+\alpha\right)  \left(
t_{s}-s\right)  \leq2\pi\left(  1+\alpha\right)  .
\]
Gathering the above identities and estimation, one gets%
\begin{align*}
\varphi_{\mathcal{U},-1}\left(  s,u\right)   &  =e^{f_{1}\left(  s\right)
}\parallel u\left(  s\right)  \parallel=e^{f_{1}\left(  s\right)
-f_{2}\left(  s\right)  }\varphi_{\mathcal{U},\alpha}\left(  s,u\right)  \leq
e^{2\pi\left(  1+\alpha\right)  }\varphi_{\mathcal{U},\alpha}\left(
s,u\right)  \text{.}%
\end{align*}
We conclude that $\mathcal{C}(\mathcal{U},\alpha)=\mathcal{C}(\mathcal{U}%
,-1)$. Since $\mathcal{C}\left(  \mathcal{U},-1\right)  \neq C_{00}%
(\mathbb{R}_{+},X)$, we get that $\mathcal{C}\left(  \mathcal{U}%
,\alpha\right)  \neq C_{00}(\mathbb{R}_{+},X)$ for every $\alpha\in
\mathcal{A}\left(  \mathcal{U}\right)  $, and by Proposition \ref{p1} it
follows that $\mathcal{U}$ is not uniform exponentially bounded.
\end{proof}

Next theorem is a fundamental result for our survey, and offers a necessary
and sufficient condition for invertibility of each generator $G_{\mathcal{U}%
,\alpha}$, $\alpha\in\mathcal{A}\left(  \mathcal{U}\right)  $.

\begin{theorem}
\label{th2}The infinitesimal generator $G_{\mathcal{U},\alpha}$ is invertible
if and only if the admissible exponent $\alpha$ is quasi-negative.
\end{theorem}

\begin{proof}
\emph{Necessity}. Assume that $G_{\mathcal{U},\alpha}$ is invertible. It
suffices to consider the case $\alpha\geq0$. According to Theorem \ref{th4},
$\mathcal{U}$ is nonuniform exponentially stable. Choose $\nu>0$ with $-\nu
\in\mathcal{A}\left(  \mathcal{U}\right)  $, $-\nu\neq\inf\mathcal{A}\left(
\mathcal{U}\right)  $. For each fixed $s\geq0$ and $x\in X\smallsetminus
\left\{  0\right\}  $ we construct the map $\widetilde{u}_{s,x}\in
C(\mathbb{R}_{+},X)$ given by%
\begin{equation}
\widetilde{u}_{s,x}\left(  \xi\right)  =%
\begin{cases}
e^{\nu\left(  \xi-s\right)  }U\left(  \xi,s\right)  x\text{,} & \text{if }%
\xi>s,\\
x\text{,} & \text{if }0\leq\xi\leq s\text{.}%
\end{cases}
\label{eq5}%
\end{equation}
Pick $\varepsilon>0$ such that $-\nu-\varepsilon\in\mathcal{A}\left(
\mathcal{U}\right)  $. For $t>s$\ we have
\begin{align*}
\varphi_{\mathcal{U},\alpha}\left(  t,\widetilde{u}_{s,x}\right)   &
=\underset{\tau\geq t}{\sup}\text{ }e^{-\alpha\left(  \tau-t\right)  }%
e^{\nu\left(  t-s\right)  }\left\Vert U\left(  \tau,s\right)  x\right\Vert \\
&  \leq\underset{\tau\geq t}{\sup}\text{ }e^{\nu\left(  t-s\right)
}e^{-\left(  \nu+\varepsilon\right)  \left(  \tau-s\right)  }M_{-\nu
-\varepsilon}\left(  s\right)  \left\Vert x\right\Vert \\
&  \leq e^{-\varepsilon\left(  t-s\right)  }M_{-\nu-\varepsilon}\left(
s\right)  \left\Vert x\right\Vert \text{,}%
\end{align*}
therefore $\underset{t\rightarrow\infty}{\lim}\varphi_{\mathcal{U},\alpha
}\left(  t,\widetilde{u}_{s,x}\right)  =0$. Proposition \ref{p2} implies that
the below set is nonempty:
\[
\Lambda_{s,x}=\left\{  t\geq0:\sup\limits_{\xi\geq0}\text{ }\varphi
_{\mathcal{U},\alpha}(\xi,\widetilde{u}_{s,x})=\varphi_{\mathcal{U},\alpha
}\left(  t,\widetilde{u}_{s,x}\right)  \right\}  .
\]
As the map $\xi\mapsto\varphi_{\mathcal{U},\alpha}(\xi,\widetilde{u}_{s,x})$
is continuous, it\ follows that $t_{s,x}=\inf\Lambda_{s,x}\in\Lambda_{s,x}$.
We have the alternative:

(A1) There exists $\nu>0$, $-\nu\in\mathcal{A}\left(  \mathcal{U}\right)  $,
$-\nu\neq\inf\mathcal{A}\left(  \mathcal{U}\right)  $ such that for each
$s\geq0$ and $x\in X\smallsetminus\left\{  0\right\}  $, $t_{s,x}\in\left[
s,s+\frac{1}{\nu}\right]  $.

(A2) For any sufficiently small $\nu>0$, $-\nu\in\mathcal{A}\left(
\mathcal{U}\right)  $, $-\nu\neq\inf\mathcal{A}\left(  \mathcal{U}\right)  $,
there exists $s\geq0$ and $x\in X\smallsetminus\left\{  0\right\}  $ with
$t_{s,x}>s+\frac{1}{\nu}$.

Assume that (A1) holds. In this case, for all $t\geq s$ and $x\in
X\smallsetminus\left\{  0\right\}  $ we have%
\begin{align*}
\parallel\widetilde{u}_{s,x}\left(  t\right)  \parallel &  \leq\varphi
_{\mathcal{U},\alpha}\left(  t,\widetilde{u}_{s,x}\right)  \leq\varphi
_{\mathcal{U},\alpha}\left(  t_{s,x},\widetilde{u}_{s,x}\right) \\
&  =\underset{\tau\geq t_{s,x}}{\sup}\text{ }e^{-\alpha\left(  \tau
-t_{s,x}\right)  }e^{\nu\left(  t_{s,x}-s\right)  }\left\Vert U\left(
\tau,s\right)  x\right\Vert \text{.}%
\end{align*}
The above estimation yields%
\begin{align*}
\underset{t\geq s}{\sup}\text{ }e^{\nu\left(  t-s\right)  }\left\Vert U\left(
t,s\right)  x\right\Vert  &  \leq\underset{\tau\geq t_{s,x}}{\sup}\text{
}e^{-\alpha\left(  \tau-t_{s,x}\right)  }e^{\nu\left(  t_{s,x}-s\right)
}\left\Vert U\left(  \tau,s\right)  x\right\Vert \\
&  =\underset{\tau\geq t_{s,x}}{\sup}\text{ }e^{\left(  \alpha+\nu\right)
\left(  t_{s,x}-s\right)  }e^{-\alpha\left(  \tau-s\right)  }\left\Vert
U\left(  \tau,s\right)  x\right\Vert \\
&  \leq e^{\frac{1}{\nu}\left(  \alpha+\nu\right)  }\underset{\tau\geq s}%
{\sup}\text{ }e^{-\alpha\left(  \tau-s\right)  }\left\Vert U\left(
\tau,s\right)  x\right\Vert \text{.}%
\end{align*}
Therefore, for all $s\geq0$ and $x\in X$\ we have%
\[
\underset{t\geq s}{\sup}\text{ }e^{-\alpha\left(  t-s\right)  }\left\Vert
U\left(  t,s\right)  x\right\Vert \leq\underset{t\geq s}{\sup}\text{ }%
e^{\nu\left(  t-s\right)  }\left\Vert U\left(  t,s\right)  x\right\Vert \leq
K\underset{t\geq s}{\sup}\text{ }e^{-\alpha\left(  t-s\right)  }\left\Vert
U\left(  t,s\right)  x\right\Vert ,
\]
where $K=e^{\frac{1}{\nu}\left(  \alpha+\nu\right)  }$. It turns out that%
\[
\varphi_{\mathcal{U},\alpha}\left(  s,u\right)  \leq\varphi_{\mathcal{U},-\nu
}\left(  s,u\right)  \leq K\varphi_{\mathcal{U},\alpha}\left(  s,u\right)
\text{, }s\geq0\text{, }u\in C(\mathbb{R}_{+},X).
\]
We conclude that $\mathcal{C}(\mathcal{U},\alpha)=\mathcal{C}(\mathcal{U}%
,-\nu)$, as Banach spaces, thus $\alpha$ is quasi-negative.

Assume that (A2) holds. For sufficiently large $n\in\mathbb{N}^{\ast}$ such
that $-\frac{1}{n}\in\mathcal{A}\left(  \mathcal{U}\right)  $, $-\frac{1}%
{n}\neq\inf\mathcal{A}\left(  \mathcal{U}\right)  $, we put $s_{n}\geq0$,
$x_{n}\in X\smallsetminus\left\{  0\right\}  $ and $t_{n}=t_{s_{n},x_{n}%
}>s_{n}+n$, as in the statement (A2). Let us define the $C^{1}$-map $\psi
_{n}:\mathbb{R}_{+}\rightarrow\mathbb{R}_{+}$ by%
\[
\psi_{n}\left(  t\right)  =%
\begin{cases}
e^{\frac{1}{n}\left(  t-s_{n}\right)  }\text{,} & \text{if }t>t_{n}\text{,}\\
\frac{t}{n}e^{\frac{1}{n}\left(  t_{n}-s_{n}\right)  }+e^{\frac{1}{n}\left(
t_{n}-s_{n}\right)  }\left(  1-\frac{t_{n}}{n}\right)  \text{,} & \text{if
}s_{n}+\delta_{n}<t\leq t_{n}\text{,}\\
a_{n}t^{2}+b_{n}t+c_{n}\text{, } & \text{if }s_{n}<t\leq s_{n}+\delta
_{n}\text{,}\\
0\text{,} & \text{if }0\leq t\leq s_{n}\text{.}%
\end{cases}
\]
The real constants $a_{n}, b_{n}, c_{n}\in\mathbb{R}$ and $\delta_{n}%
\in\left(  s_{n},t_{n}\right)  $ are determined by the $C^{1}$-condition
imposed to the map $\psi_{n}$. We also define $u_{n}$, $f_{n}:\mathbb{R}%
_{+}\rightarrow X$ by%
\[
u_{n}\left(  t\right)  =%
\begin{cases}
\psi_{n}\left(  t\right)  U\left(  t,s_{n}\right)  x_{n}\text{,} & \text{if
}t>s_{n}\text{,}\\
0\text{,} & \text{if }0\leq t\leq s_{n}\text{,}%
\end{cases}
\]
and%
\[
f_{n}\left(  t\right)  =%
\begin{cases}
\psi_{n}^{\prime}\left(  t\right)  U\left(  t,s_{n}\right)  x_{n}\text{,} &
\text{if }t>s_{n}\text{,}\\
0\text{,} & \text{if }0\leq t\leq s_{n}\text{.}%
\end{cases}
\]
Evidently $u_{n},$ $f_{n}\in\mathcal{C}(\mathcal{U},\alpha)$ and
$G_{\mathcal{U},\alpha}u_{n}=-f_{n}$. Notice that for $t\in\left[
0,t_{n}\right]  $ we have%
\begin{align*}
\varphi_{\mathcal{U},\alpha}\left(  t,u_{n}\right)   &  =\underset{\tau\geq
t}{\sup}\text{ }e^{-\alpha\left(  \tau-t\right)  }\psi_{n}\left(  t\right)
\left\Vert U\left(  \tau,s_{n}\right)  x_{n}\right\Vert \\
&  \leq\underset{\tau\geq t}{\sup}\text{ }e^{-\alpha\left(  \tau-t\right)
}e^{\frac{1}{n}\left(  t-s_{n}\right)  }\left\Vert U\left(  \tau,s_{n}\right)
x_{n}\right\Vert \\
&  =\varphi_{\mathcal{U},\alpha}\left(  t,\widetilde{u}_{s_{n},x_{n}}\right)
\text{,}%
\end{align*}
where $\widetilde{u}_{s_{n},x_{n}}$ is defined in (\ref{eq5}). If $t\geq
t_{n}$, as $\varphi_{\mathcal{U},\alpha}\left(  t,u_{n}\right)  =\varphi
_{\mathcal{U},\alpha}\left(  t,\widetilde{u}_{s_{n},x_{n}}\right)  $, it
follows that $\parallel u_{n}\parallel_{\mathcal{U},\alpha}=\varphi
_{\mathcal{U},\alpha}\left(  t_{n},\widetilde{u}_{s_{n},x_{n}}\right)  $, and
similarly one gets $$\parallel f_{n}\parallel_{\mathcal{U},\alpha}=\frac{1}%
{n}\varphi_{\mathcal{U},\alpha}\left(  t_{n},\widetilde{u}_{s_{n},x_{n}%
}\right).$$
These identities imply%
\[
\frac{\parallel u_{n}\parallel_{\mathcal{U},\alpha}}{\parallel f_{n}%
\parallel_{\mathcal{U},\alpha}}=n\text{,}%
\]
which leads to $\underset{f\in\mathcal{C}(\mathcal{U},\alpha)}{\sup}$
$\frac{\left\Vert G_{\mathcal{U},\alpha}^{-1}f\right\Vert _{\mathcal{U}%
,\alpha}}{\left\Vert f\right\Vert _{\mathcal{U},\alpha}}\geq n$. Thus,
$G_{\mathcal{U},\alpha}^{-1}$ is not bounded, which is false, and eventually
only alternative (A1) holds, that proves the claim.

\emph{Sufficiency}. Assume that $\alpha\in\mathcal{A}\left(  \mathcal{U}%
\right)  $ is a quasi-negative exponent, and choose $\beta\in\mathcal{A}%
\left(  \mathcal{U}\right)  $, $\beta<0$ such that $\mathcal{C}(\mathcal{U}%
,\alpha)=\mathcal{C}(\mathcal{U},\beta)$. Remark \ref{rem.gen} implies
$G_{\mathcal{U},\alpha}=G_{\mathcal{U},\beta}$. Since $\mathcal{U}$ is $\beta
$-nonuniform exponentially stable, from the proof of the necessity of Theorem
\ref{th4} it follows that $G_{\mathcal{U},\beta}$ is invertible as well as
$G_{\mathcal{U},\alpha}$.
\end{proof}

For any $C_{0}$-semigroup $\mathcal{T}=\left\{  T(t)\right\}  _{t\geq0}$ with
generator $G$, the \emph{spectral mapping inclusion }holds (see Theorem 2.6 in
\cite{Ch.La.1999}), that is
\[
e^{t\sigma(G)}\subset\sigma(T(t))\text{, }t\geq0.
\]
In the particular case of the evolution semigroups (that is in the uniform
setting) this feature is stronger, precisely the above inclusion also holds
the other way, and this is what is called the \emph{spectral mapping theorem}.
Our next result emphasizes that this important property of evolution
semigroups is still valid in the nonuniform case.

\begin{theorem}
Each $\mathcal{T}_{\alpha}$, $\alpha\in\mathcal{A}\left(  \mathcal{U}\right)
$, satisfies the identity%
\[
e^{t\sigma\left(  G_{\mathcal{U},\alpha}\right)  }=\sigma\left(  T_{\alpha
}\left(  t\right)  \right)  \smallsetminus\left\{  0\right\}  \text{, }%
t\geq0\text{.}%
\]
Moreover, $\sigma\left(  G_{\mathcal{U},\alpha}\right)  $ is a left half-plane
and $\sigma\left(  T_{\alpha}\left(  t\right)  \right)  $, $t\geq0$, is a disc.
\end{theorem}

\begin{proof}
If $\lambda\in\mathbb{C}$, then we consider the \emph{rescaled evolution
family}
\[
U_{\lambda}(t,s)=e^{-\lambda(t-s)}U(t,s).
\]
For arbitrary $\alpha\in\mathcal{A}\left(  \mathcal{U}\right)  $, $t\geq0$ and
$u\in C(\mathbb{R}_{+},X)$ one has
\begin{align*}
\varphi_{\mathcal{U}_{\lambda},\alpha-\operatorname{Re}\lambda}\left(
t,u\right)   &  =\underset{\tau\geq t}{\sup}\text{ }e^{-\left(  \alpha
-\operatorname{Re}\lambda\right)  \left(  \tau-t\right)  }\left\Vert
U_{\lambda}\left(  \tau,t\right)  u\left(  t\right)  \right\Vert \\
&  =\underset{\tau\geq t}{\sup}\text{ }e^{-\alpha\left(  \tau-t\right)
}\left\Vert U\left(  \tau,t\right)  u\left(  t\right)  \right\Vert
=\varphi_{\mathcal{U},\alpha}\left(  t,u\right)  \text{.}%
\end{align*}
We deduce that $\mathcal{A}\left(  \mathcal{U}_{\lambda}\right)
=\mathcal{A}\left(  \mathcal{U}\right)  -\operatorname{Re}\lambda$, and
$\mathcal{C}(\mathcal{U}_{\lambda},\alpha-\operatorname{Re}\lambda
)=\mathcal{C}(\mathcal{U},\alpha)$. The definition of the generator implies
also that $D\left(  G_{\mathcal{U},\alpha}\right)  =D\left(  G_{\mathcal{U}%
_{\lambda},\alpha-\operatorname{Re}\lambda}\right)  $ and
\begin{equation}
G_{\mathcal{U}_{\lambda},\alpha-\operatorname{Re}\lambda}=G_{\mathcal{U}%
,\alpha}-\lambda\,\mathrm{Id}\text{.} \label{eq1}%
\end{equation}
Let $\lambda\in\rho\left(  G_{\mathcal{U},\alpha}\right)  $, that is
$G_{\mathcal{U}_{\lambda},\alpha-\operatorname{Re}\lambda}$ is invertible.
Choose $\mu\in\mathbb{C}$ with $\operatorname{Re}\mu\geq\operatorname{Re}%
\lambda$. We only have two options:

(B1) $\operatorname{Re}\mu>\alpha$;

(B2) $\operatorname{Re}\lambda\leq\operatorname{Re}\mu\leq\alpha$.

In case (B1), Theorem \ref{th2} and formula (\ref{eq1}) imply that
$G_{\mathcal{U}_{\mu},\alpha-\operatorname{Re}\mu}$ is invertible,
consequently $\mu\in\rho\left(  G_{\mathcal{U},\alpha}\right)  $.

Assume that (B2) holds. According to the same Theorem \ref{th2}, as
$G_{\mathcal{U}_{\lambda},\alpha-\operatorname{Re}\lambda}$ is invertible,
there exists $\nu>0$ with $\mathcal{C}(\mathcal{U}_{\lambda},\alpha
-\operatorname{Re}\lambda)=\mathcal{C}(\mathcal{U}_{\lambda},-\nu)$, and
\[
G_{\mathcal{U}_{\lambda},\alpha-\operatorname{Re}\lambda}=G_{\mathcal{U}%
_{\lambda},-\nu}.
\]
Using Eq. (\ref{eq1}), we successively have%
\begin{align*}
G_{\mathcal{U}_{\mu},\alpha-\operatorname{Re}\mu}  &  =G_{\mathcal{U}%
_{\lambda},\alpha-\operatorname{Re}\lambda}+\left(  \lambda-\mu\right)
\,\mathrm{Id}\\
&  =G_{\mathcal{U}_{\lambda},-\nu}+\left(  \lambda-\mu\right)  \,
\mathrm{Id}\\
&  =G_{\mathcal{U}_{\mu},-\nu+\operatorname{Re}\lambda-\operatorname{Re}\mu
}\text{.}%
\end{align*}
Since $-\nu+\operatorname{Re}\lambda-\operatorname{Re}\mu<0$, then
$G_{\mathcal{U}_{\mu},-\nu+\operatorname{Re}\lambda-\operatorname{Re}\mu}$ is
invertible, hence $G_{\mathcal{U}_{\mu},\alpha-\operatorname{Re}\mu}$ is also
invertible. This shows that $\mu\in\rho\left(  G_{\mathcal{U},\alpha}\right)
$, which implies that the spectrum $\sigma\left(  G_{\mathcal{U},\alpha
}\right)  $ is a left half-plane. For the rest of the proof we refer the reader to
\cite{Mi.Ra.Sc.1998}.
\end{proof}

\section{Nonuniform exponential stability of bounded orbits}

\label{sec.bounded}

From Section \ref{sec.nonuniform} we already know that a spectral property of
some generator $G_{\mathcal{U},\alpha}$ ( $0\notin\sigma(G_{\mathcal{U}%
,\alpha})$), can completely characterize the existence of nonuniform
exponential stability for the evolution family $\mathcal{U}$. To be useful in practical applications, a linear dynamical system must provide information not only on the global stability of the dynamics, but also on the stability of a given trajectory. In what follows
we expose a sufficient condition for the existence of nonuniform exponential
stability of all bounded orbits of $\mathcal{U}$. To be more specific, we
replace hypothesys $0\notin\sigma(G_{\mathcal{U},\alpha})$, by a weaker one
imposed this time to the approximate point spectrum of $G_{\mathcal{U},\alpha
}$. Let us notice that this result generalizes in fact Theorem 3.2 in
\cite{Mi.Ra.Sc.1998}. Recall that the \emph{approximate point spectrum} of a
bounded linear operator $T:X\rightarrow X$ is the set
\[
\sigma_{ap}(T)=\{\lambda\in\mathbb{C}:\,\forall\,\varepsilon>0\;\exists\,x\in
X\text{ , }\Vert x\Vert=1\text{, with }\Vert\lambda x-Tx\Vert\leq
\varepsilon\}.
\]

\begin{theorem}
\label{th.aprox} Let $\mathcal{U}$ be a nonuniform exponentially bounded
evolution family, and $\alpha\in\mathcal{A}\left(  \mathcal{U}\right)  $ with
property
\[
\sigma_{ap}(G_{\mathcal{U},\alpha})\cap i\mathbb{R}\neq i\mathbb{R}.
\]
If $\sup\limits_{t\geq t_{0}}\parallel U(t,t_{0})x_{0}\parallel<\infty$ for
some $x_{0}\in X$ and $t_{0}\geq0$, then there exist a constant $\nu>0$ and a
continuous function $N:\mathbb{R}_{+}\rightarrow(0,\infty)$, independent of
$x$ and $t_{0}$, satisfying
\[
\parallel U(t,t_{0})x_{0}\parallel\leq N(t_{0})e^{-\nu(t-t_{0})}\parallel
x_{0}\parallel,\text{ for all }t\geq t_{0}. \label{eq.aprox}%
\]

\end{theorem}

\begin{proof}
It suffices to consider $\alpha\geq0$. One can easily prove that Lemma 3.1 in
\cite{Mi.Ra.Sc.1998} also holds in our case, and thus $0\notin\sigma
_{ap}(G_{\mathcal{U},\alpha})$. This implies that there exists $c=c(\alpha)>0$
such that
\[
\label{eq.th2.1}\parallel u\parallel_{\mathcal{U},\alpha}\leq c \parallel
G_{\mathcal{U},\alpha} u\parallel_{\mathcal{U},\alpha}, \text{ for all } u\in
D(G_{\mathcal{U},\alpha}).
\]
Let $x_{0}\in X$ and $t_{0}\geq0$ such that $\sup\limits_{\tau\geq t_{0}%
}\parallel U(\tau, t_{0})x_{0}\parallel<\infty$. The proof is divided into
four steps.

\textit{Step 1.} We prove that%
\begin{equation}
\label{eq.t2.s1}\parallel U(t,t_{0})x_{0} \parallel\leq\frac{2c}{t-t_{0}}%
\sup\limits_{\tau\geq t_{0}}\parallel U(\tau, t_{0})x_{0}\parallel, \text{ for
} t>t_{0}.
\end{equation}

For any fixed $t>0$, we define $\psi_{t}:\mathbb{R}_{+}\to\mathbb{R}_{+}$ by
\[
\psi_{t}(s)=
\begin{cases}
\left(  \frac{4}{3}\right)  ^{3/2}t^{-{7}/{2}}s^{{3}/{2}}, & \text{ if }
s\in\left[  0,\frac{3t}{4}\right)  ,\\
-\frac{8}{t^{4}}s^{2}+\frac{14}{t^{3}}s-\frac{5}{t^{2}}, & \text{ if }
s\in\left[  \frac{3t}{4},t\right)  ,\\
\frac{1}{s^{2}}, & \text{ if } s\geq t.
\end{cases}
\]
An easy computation shows that
\[
\psi_{t}(t)=\frac{1}{t^{2}}, \text{ } \sup\limits_{s\geq0}\psi_{t}(s)=\frac
{9}{8t^{2}}, \text{ and } \lim\limits_{s\to\infty}\psi_{t}(s)=0.
\]
Moreover, $\psi_{t}$ is a function of class $C^{1}$ on $\mathbb{R}_{+}$ with
$\sup\limits_{s\geq0}|\psi_{t}^{\prime}(s)|=\frac{2}{t^{3}}$ and $\psi
_{t}^{\prime}(0)=0$.

Fix now $t>t_{0}$. We consider a function $g:\mathbb{R}_{+} \to\mathbb{R}_{+}$
defined by
\[
g(s)=
\begin{cases}
\psi_{t-t_{0}}(s-t_{0}), & \text{ if } s\geq t_{0},\\
0, & \text{ if } s\in[0,t_{0}).
\end{cases}
\]
Let us also define $u,f:\mathbb{R}_{+}\to X$ by
\[
u(s)=
\begin{cases}
g(s)U(s,t_{0})x_{0}, & \text{ if } s\geq t_{0},\\
0, & \text{ if } s\in[0,t_{0}),
\end{cases}
\]
and
\[
f(s)=
\begin{cases}
g^{\prime}(s)U(s,t_{0})x_{0}, & \text{ if } s\geq t_{0},\\
0, & \text{ if } s\in[0,t_{0}).
\end{cases}
\]
We claim that $u\in D(G_{\mathcal{U},\alpha})$, $f\in\mathcal{C}(\mathcal{U},
\alpha)$, and $G_{\mathcal{U},\alpha}u=-f$. Since
\[
u(s)=\int_{0}^{s} U(s,\xi)f(\xi)d\xi, \text{ for } s\geq0,
\]
it suffices to prove that $u,f\in\mathcal{C}(\mathcal{U}, \alpha)$. For $s\geq
t_{0}$ we obtain
\begin{align*}
\varphi_{\mathcal{U},\alpha}(s,u)  &  =\sup\limits_{\tau\geq s} e^{-\alpha
(\tau-s)} g(s) \parallel U(\tau,t_{0})x_{0}\parallel\\
&  \leq\sup\limits_{\tau\geq t_{0}} \parallel U(\tau,t_{0})x_{0}\parallel
g(s)\\
&  \leq\frac{9}{8(s-t_{0})^{2}} \sup\limits_{\tau\geq t_{0}} \parallel
U(\tau,t_{0})x_{0}\parallel\to0 \text{ as } s\to\infty,
\end{align*}
and thus $u\in\mathcal{C}(\mathcal{U}, \alpha)$. Using the same lines, one can
prove that $f\in\mathcal{C}(\mathcal{U}, \alpha)$, and
\[
\parallel f\parallel_{\mathcal{U},\alpha} \leq\frac{2}{(t-t_{0})^{3}}
\sup\limits_{\tau\geq t_{0}} \parallel U(\tau,t_{0})x_{0}\parallel.
\]
Now we have
\begin{align*}
\parallel u(t)\parallel &  =\frac{1}{(t-t_{0})^{2}} \parallel U(t,t_{0}%
)x_{0}\parallel\leq\varphi_{\mathcal{U},\alpha}(t,u)\leq\parallel
u\parallel_{\mathcal{U},\alpha} \leq c \parallel f\parallel_{\mathcal{U}%
,\alpha}\\
&  \leq\frac{2c}{(t-t_{0})^{3}} \sup\limits_{\tau\geq t_{0}} \parallel
U(\tau,t_{0})x_{0}\parallel,
\end{align*}
meaning that \eqref{eq.t2.s1} holds.

\textit{Step 2.} We show that there exists a constant $K=K(\alpha)>0$ such
that%
\begin{equation}
\label{eq.t2.s2}\parallel U(t,t_{0})x_{0} \parallel\leq\frac{cK}{t-t_{0}}M
(t_{0})\parallel x_{0}\parallel, \text{ for }t>t_{0}.
\end{equation}

By \eqref{eq.t2.s1} we get $\lim\limits_{t\to\infty}\parallel U(t,t_{0}%
)x_{0}\parallel=0$, and since the function $t\mapsto\parallel U(t,t_{0}%
)x_{0}\parallel$ is continuous, we have
\[
\sup\limits_{\tau\geq t_{0}} \parallel U(\tau,t_{0})x_{0}\parallel=\parallel
U(t_{0}^{\prime},t_{0})x_{0}\parallel, \text{ for some } t_{0}^{\prime}\geq
t_{0}.
\]
Setting $t=t_{0}^{\prime}$ in \eqref{eq.t2.s1}, we get $t_{0}^{\prime}%
-t_{0}\leq2c $ and from Definition \ref{d.exp.bound}, we obtain
\begin{align*}
\parallel U(t,t_{0})x_{0}\parallel &  \leq\frac{2c}{t-t_{0}}\parallel
U(t_{0}^{\prime}, t_{0})x_{0}\parallel\leq\frac{2c}{t-t_{0}}e^{\alpha
(t_{0}^{\prime}-t_{0})} M(t_{0}) \parallel x_{0}\parallel\\
&  \leq\frac{2c}{t-t_{0}}e^{2\alpha c} M(t_{0}) \parallel x_{0}\parallel.
\end{align*}
We thereby conclude that \eqref{eq.t2.s2} holds for $K=2e^{2\alpha c}$.

\textit{Step 3.} We prove that for all $n\in\mathbb{N}^{\ast}$,
\begin{equation}
\label{eq.t2.s3}\parallel U(t,t_{0})x_{0}\parallel\leq\frac{n!c^{n}K}%
{(t-t_{0})^{n}}M(t_{0})\parallel x_{0}\parallel, \text{ for }t>t_{0}.
\end{equation}

Assume that the previous inequality works for some positive integer $n-1$. For
any $t>0$ we define a function $\psi_{n,t}:\mathbb{R}_{+}\to\mathbb{R}_{+}$
by
\[
\psi_{n,t}(s)=
\begin{cases}
s^{n} , & \text{ if } s\in\left[  0,t\right)  ,\\
As^{2}+Bs+C, & \text{ if } s\in\left[  t,t+h\right)  ,\\
\frac{D}{s^{2}}, & \text{ if } s\geq t+h,
\end{cases}
\]
where the constants $A$, $B$, $C$, $D$ and $h$ are chosen such that
$\psi_{n,t}$ be a function of class $C^{1}$ on $\mathbb{R}_{+}$.

Fix $t>t_{0}$. Using the function above we construct $g_{n}:\mathbb{R}_{+}
\to\mathbb{R}_{+}$ defined by
\[
g_{n}(s)=
\begin{cases}
\psi_{n,t-t_{0}}(s-t_{0}), & \text{ if } s\geq t_{0},\\
0, & \text{ if } s\in[0,t_{0}).
\end{cases}
\]
Like in Step 1 we put
\[
u_{n}(s)=
\begin{cases}
g_{n}(s)U(s,t_{0})x_{0}, & \text{ if } s\geq t_{0},\\
0, & \text{ if } s\in[0,t_{0}),
\end{cases}
\]
and
\[
f_{n}(s)=
\begin{cases}
g_{n}^{\prime}(s)U(s,t_{0})x_{0}, & \text{ if } s\geq t_{0},\\
0, & \text{ if } s\in[0,t_{0}).
\end{cases}
\]
As previously one can prove that $u_{n}, f_{n}\in\mathcal{C}(\mathcal{U},
\alpha)$, in order to obtain $u_{n}\in D(G_{\mathcal{U},\alpha})$, and
$G_{\mathcal{U},\alpha}u_{n}=-f_{n}$. For $s\geq t_{0}$ we also have
\begin{align*}
\varphi_{\mathcal{U},\alpha}(s,f_{n})  &  \leq| g_{n}^{\prime}(s)|
\sup\limits_{\tau\geq s} \parallel U(\tau,t_{0})x_{0}\parallel\\
&  \leq n (s-t_{0})^{n-1}\sup\limits_{\tau\geq s}\frac{(n-1)!c^{n-1}K}%
{(\tau-t_{0})^{n-1}}M(t_{0})\parallel x_{0}\parallel\\
&  \leq n (s-t_{0})^{n-1}\frac{(n-1)!c^{n-1}K}{(s-t_{0})^{n-1}}M(t_{0}%
)\parallel x_{0}\parallel\\
&  =n! c^{n-1} K M(t_{0})\parallel x_{0}\parallel.
\end{align*}
This inequality implies $\parallel f_{n}\parallel_{\mathcal{U},\alpha}\leq n!
c^{n-1} K M(t_{0})\parallel x_{0}\parallel$. Then we get
\[
\parallel u_{n}(t)\parallel=(t-t_{0})^{n}\parallel U(t,t_{0})x_{0}%
\parallel\leq\parallel u_{n}\parallel_{\mathcal{U},\alpha}\leq c \parallel
f_{n}\parallel_{\mathcal{U},\alpha}\leq n! c^{n} K M(t_{0})\parallel
x_{0}\parallel,
\]
that proves \eqref{eq.t2.s3}.

\textit{Step 4.} For any fixed $\delta>c$, we show that
\begin{equation}
\label{eq.t2.s4}\parallel U(t,t_{0}) x_{0}\parallel\leq\frac{\delta K}%
{\delta-c}e^{-\frac{1}{\delta}(t-t_{0})}M(t_{0})\parallel x_{0}\parallel,
\text{ for } t> t_{0}.
\end{equation}

Using the notations in Step 2, we have
\begin{align*}
\parallel U(t,t_{0})x_{0}\parallel &  \leq\sup\limits_{\tau\geq t_{0}%
}\parallel U(\tau,t_{0})x_{0}\parallel=\parallel U(t_{0}^{\prime},t_{0}%
)x_{0}\parallel\\
&  \leq M(t_{0})e^{\alpha(t_{0}^{\prime}-t_{0})}\parallel x_{0}\parallel\leq
KM(t_{0})\parallel x_{0}\parallel,\text{ }t\geq t_{0}.
\end{align*}
Choose $\delta>c$. Dividing \eqref{eq.t2.s3} by $\delta^{n}$ and using the
above inequality one gets
\[
\frac{\left[  \delta^{-1}(t-t_{0})\right]  ^{n}}{n!}\parallel U(t,t_{0}%
)x_{0}\parallel\leq K\left(  \frac{c}{\delta}\right)  ^{n}M(t_{0})\parallel
x_{0}\parallel,\;t\geq t_{0},\,n\in\mathbb{N}.
\]
Summing with respect to $n$ we easily obtain formula (\ref{eq.t2.s4}).
\end{proof}

\section{Comments and Final Remarks}

An important comment is needed. If an evolution family is uniform
exponentially bounded, then it is possible to define its corresponding
evolution semigroup on some generic function spaces, as for example
$C_{00}(\mathbb{R}_{+},X)$. We notice that the space $C_{00}(\mathbb{R}%
_{+},X)$\emph{ models all the uniform exponentially bounded evolution
families}. For instance, in \cite{Mi.Ra.Sc.1998} it is shown that the
existence of uniform exponential stability is equivalent to the invertibility
of the generator of the evolution semigroup on $C_{00}(\mathbb{R}_{+},X)$.
After lecturing our paper the reader will certainly remark that in the nonuniform
setting, usually \emph{the modelling function space is not unique}. In fact
these spaces (that we call admissible) are merely dense subspaces of
$C_{00}(\mathbb{R}_{+},X)$, depending as explained above on each evolution
family, and on each particular admissible exponent. Let us also mention that our
Example \ref{ex2} (which is a version of a classical example of Perron, see \cite[p. 123]{Dal}) presents an evolution family with \emph{infinitely
many admissible Banach function spaces}.
We hope that our constructions and arguments will contribute to a better
understanding of the complexity and unpredictability of the nonuniform
behavior. We also hope that our results will lead to extending some known
results in the theory of exponential dichotomies, from the uniform to the
nonuniform case. As it seems, for such endeavor one may probably need to use
different admissible Banach function spaces (see Theorem 4.5 in
\cite{Mi.Ra.Sc.1998}).

\section*{Acknowledgments} The work of N.L. was partially supported by Horizon2020-2017-RISE-777911 project.

\end{document}